\documentclass[11pt]{amsart}

\usepackage{amsmath, amsthm, graphics,setspace,xypic,hyperref}
\usepackage{latexsym}
\usepackage{mathtools}
\usepackage{eucal}
\usepackage{caption}

\usepackage{amssymb}
\usepackage{color}
\usepackage{amscd}
\usepackage{palatino}
\usepackage{latexsym}
\usepackage{epsfig}
\usepackage{graphicx}
\usepackage{amsfonts} 
\usepackage{psfrag}
\usepackage[all]{xy}
\usepackage{youngtab, ytableau}
\usepackage{url}
\usepackage{cite}

\usepackage[margin=1.3in]{geometry}

\definecolor{ltgrey}{RGB}{180, 187, 198}

\input xy
\xyoption{all}
\UseComputerModernTips

\numberwithin{equation}{section}
\numberwithin{figure}{section}

\newtheorem{THM}{Theorem}
\newtheorem{PROP}[THM]{Proposition}
\newtheorem{COR}[THM]{Corollary}

\newtheorem{theorem}{Theorem}[section]
\newtheorem{lemma}[theorem]{Lemma}
\newtheorem{proposition}[theorem]{Proposition}

\newtheorem{corollary}[theorem]{Corollary}

\newtheorem{remark}[theorem]{Remark}
\newtheorem{example}[theorem]{Example}

\theoremstyle{definition}
\newtheorem{definition}[theorem]{Definition}

\newcommand{\C}{{\mathbb{C}}}

\newcommand{\N}{{\mathbb{N}}}

\renewcommand{\P}{{\mathbb{P}}}

\newcommand{\B}{\mathcal{B}}

\newcommand{\g}{\mathfrak{g}}

\renewcommand{\b}{\mathfrak{b}}

\def\br{\mathsf{Br}}

\def\Inv{\mathsf{I}\mathsf{n}\mathsf{v}}
\def\invv{\mathsf{i}\mathsf{n}\mathsf{v}}

\def\gg{{\mathfrak g}}
\def\bb{{\mathfrak b}}

\def\gl{{\mathfrak{gl}}}
\def\sl{{\mathfrak{sl}}}

\DeclareMathOperator{\Ima}{Im}

\usepackage{multicol}
\usepackage{color} 
\definecolor{gold}{rgb}{0.85,.66,0}
\definecolor{cherry}{rgb}{0.9,.1,.2}
\definecolor{burgundy}{rgb}{0.8,.2,.2}
\definecolor{orangered}{rgb}{0.85,.3,0}
\definecolor{orange}{rgb}{0.85,.4,0}
\definecolor{olive}{rgb}{.45,.4,0}
\definecolor{lime}{rgb}{.6,.9,0}
\definecolor{green}{rgb}{.2,.7,0}
\definecolor{grey}{rgb}{.4,.4,.2}
\definecolor{brown}{rgb}{.4,.3,.1}

\newcommand{\fb}{\mathfrak{b}}

\def\inv{\mathsf{Inv}}

\def\poin{\mathsf{Poin}}

\def\ff{{\mathbb F}}
\def\sl{\mathfrak{sl}}
\def\m{{\mathbf m}}
\def\mm{\m_{\bf{max}}}
\def\NN{{\mathcal N}}

\def\lb{\left[}
\def\rb{\right]}

\def\zzz{{\mathbf z}}

\def\ii{{\text i}{\text n}_\preceq}

\def\iii{{\text i}{\text n}_{\preceq'}}

\def\br{\mathsf{br}}

\def\gl{\mathfrak{gl}}

\DeclareMathOperator{\nilp}{\mathsf{n}}
\DeclareMathOperator{\semi}{\mathsf{s}}

\begin{document}

\title[Hessenberg varieties of codimension one in the flag variety]{Hessenberg varieties of codimension one in the flag variety}

\author{Laura Escobar}
\address{Mathematics Department\\ University of California, Santa Cruz \\ 1156 High Street \\ Santa Cruz, California   95064 \\ U.S.A. }
\email{lauraescobar@ucsc.edu}

\author{Martha Precup}
\address{Department of Mathematics\\ Washington University in St. Louis \\ One Brookings Drive \\ St. Louis, Missouri  63130 \\ U.S.A. }
\email{martha.precup@wustl.edu}

\author{John Shareshian}
\address{Department of Mathematics\\ Washington University in St. Louis \\ One Brookings Drive \\ St. Louis, Missouri  63130 \\ U.S.A. }
\email{jshareshian@wustl.edu}

\begin{abstract}  We study geometric and topological properties of Hessenberg varieties of codimension one in the type A flag variety.  Our main results: (1) give a formula for the Poincar\'e polynomial, (2) characterize when these varieties are irreducible, and (3) show that all are reduced schemes.  We prove that the singular locus of any nilpotent codimension one Hessenberg variety is also a Hessenberg variety. A key tool in our analysis is a new result applying to all (type A) Hessenberg varieties without any restriction on codimension, which states that their Poincar\'e polynomials can be computed by counting the points in the corresponding variety defined over a finite field. 
The results below were originally motivated by work of the authors in~[MR4960071] studying the precise relationship between Hessenberg and Schubert varieties, and we obtain a corollary extending the results from that paper to all codimension one (type A) Schubert varieties.
\end{abstract}

\maketitle

\section{Introduction}

We consider Hessenberg varieties of type A.  
Given a positive integer $n$, let $G=GL_n(\C)$ and $B \subseteq G$ the Borel subgroup of upper triangular matrices.  Let $\g=\mathfrak{gl}_n(\C)$ and $\b$ be the respective Lie algebras of $G$ and $B$.  We consider the action of $G$ on $\g$ by conjugation.  Given $\mathsf{x} \in \g$ and a $B$-invariant subspace $H \subseteq \g$, the type A \textit{Hessenberg variety} $\B(\mathsf{x},H)$ consists of all $gB$ in the flag variety $\B:=G/B$ satisfying $g^{-1}\mathsf{x}g \in H$.  

We assume throughout that $H$ contains the Borel subalgebra $\b$, as was assumed by De Mari and Shayman in \cite{DS1988} and by De Mari, Procesi and Shayman in \cite{DPS1992}, seminal works on Hessenberg varieties.  Under this assumption, there is a unique \textit{Hessenberg vector} $\m=(\m(1),\ldots,\m(n))$, that is, a weakly increasing sequence from $[n]$ satisfying $\m(i) \geq i$ for all $i \in [n]$, such that $H$ consists of all matrices $\mathsf{a} = (\mathsf{a}_{ij}) \in \g$ such that $\mathsf{a}_{ij}=0$ if $i>\m(j)$.  We write $H=H(\m)$ in this case.    We omit the assumption from \cite{DS1988,DPS1992} that $\mathsf{x}$ is regular and semisimple.  

Of particular interest herein is the case where $H=H(\mm)$ with
$$
\mm:=(n-1,n,\ldots,n).
$$
Notice that $H(\mm)$ is the unique maximal proper $B$-invariant subspace of $\g$, hence our choice of notation.
We observe that if $\mathsf{x}$ is not scalar, then $\B(\mathsf{x},H(\mm))$ has codimension one in $\B$.  Conversely, we showed in \cite{EPS} that if $\B(\mathsf{x},H)$ has codimension one in $\B$, then $H=H(\mm)$ or there is some $\lambda \in \C$ such that $\mathsf{x}-\lambda I_n$ has rank one.  We continue our study of codimension one Hessenberg varieties in this manuscript, concentrating on $\B(\mathsf{x},H(\mm))$.  

One of our main tools is a result that applies with no condition on codimension and could be of use in other settings. 
We define as usual the \textit{Poincar\'e polynomial} of $\B(\mathsf{x}, H)$ by
$$
\poin(\B(\mathsf{x},H);q):=\sum_j \dim_\C(H^j(\B(\mathsf{x},H);\C))q^j.
$$
Given a Hessenberg vector $\m$, the variety $\B(\mathsf{x},H(\m))$ admits an affine paving, as shown by Tymoczko in \cite{Tymoczko2006}.  It follows that the (singular) cohomology of $\B(\mathsf{x},H(\m))$ is concentrated in even degrees.  So  $\poin(\B(\mathsf{x},H(\m));q^{1/2})$ is a polynomial in $q$.  It follows from the particulars of Tymoczko's result that, given any $\mathsf{x}$ and $H(\m)$ as above, there exists an upper triangular integer matrix $\mathsf{x}^\prime$ such that 
$$
\poin(\B(\mathsf{x},H(\m));q)=\poin(\B(\mathsf{x}^\prime,H(\m));q).
$$
We may assume that $\mathsf{x}^\prime$ is in what Tymoczko calls highest form and permuted Jordan form, which will be defined later.  In addition, we may assume that the eigenvalues of x are integers.  In particular, $\mathsf{x}^\prime$ is the sum of a diagonal integer matrix $\semi$ and a strictly upper triangular $0-1$ matrix $\nilp$.  We write $m_{\semi}$ for the largest absolute value of an entry of $\semi$.  Given any prime $p$, let $\bar{\mathsf{x}}^\prime \in \gl_n(\ff_p)$ be the matrix obtained from $\mathsf{x}^\prime$ by reducing each entry modulo $p$.  Let $\bar{G}=GL_n(\ff_p)$ and let $\bar{B}$ be the subgroup of $\bar{G}$ consisting of upper triangular matrices.  The Hessenberg vector $\m$ determines a subspace $\bar{H}(\m)$ of the Lie algebra $\gl_n(\ff_p)$ in the same way it determines $H(\m)$.  We define
$$
\B_p(\bar{\mathsf{x}}^\prime ,\bar{H}(\m)):=\{g\bar{B} \in \bar{G}/\bar{B}\mid g^{-1}\bar{\mathsf{x}}^\prime g \in \bar{H}(\m)\}.
$$ 

\begin{THM}[The point count heuristic] \label{qcount}
If $p>2m_{\semi}$, then
\begin{equation} \label{qcounteq}
\left|\B_p(\bar{\mathsf{x}}^\prime ,\bar{H}(\m))\right|=\poin(\B(\mathsf{x}^\prime,H(\m));p^{1/2}).
\end{equation}
\end{THM}

We prove Theorem~\ref{qcount} in Section \ref{qcproof.sec} through close examination of Tymoczko's affine paving construction.  The point count heuristic allows us to compute the Poincar\'e polynomial of $\B(\mathsf{x}, H(\m))$ using elementary methods. For example, we obtain a simple formula for the Poincar\'e polynomial of $\B(\mathsf{x},H(\mm))$ for arbitrary $\mathsf{x}$.   It is not hard to enumerate $\B_p(\mathsf{y},\bar{H}(\mm))$ for arbitrary $\mathsf{y} \in \gl_n(\ff_p)$.    The following result appears in Section \ref{counting.sec} below.

\begin{PROP} \label{eigenspaces}
If $\mathsf{y} \in \gl_n(\ff_p)$ fixes exactly $k$ $1$-dimensional subspaces of $\ff_p^n$, then
$$
\left|\B_p(\mathsf{y},\bar{H}(\mm))\right|=[n-2]_p!\left([n]_p[n-2]_p+kp^{n-2}\right).
$$
\end{PROP}

The next result follows quickly from Theorem \ref{qcount} and Proposition \ref{eigenspaces}.

\begin{THM} \label{poinpoly}
If $\mathsf{x} \in \g$ has exactly $\ell$ pairwise distinct eigenvalues $\lambda_1,\ldots \lambda_\ell$ and $\dim_\C\ker(\mathsf{x}-\lambda_jI_n)=d_j$ for each $j \in [\ell]$, then the Poincar\'e polynomial of $\B(\mathsf{x},H(\mm))$ is
\begin{equation} \label{poinpolyform}
\poin(\B(\mathsf{x},H(\mm));q)=[n-2]_{q^2}!\left([n]_{q^2}[n-2]_{q^2}+q^{2n-4}\sum_{j=1}^\ell [d_j]_{q^2}\right)
\end{equation}
\end{THM}

Our proof of \eqref{poinpolyform} using the point count heuristic is much simpler than applying previously known Betti number formulas from, e.g.,~\cite{Tymoczko2006, Precup2013}, which each require multiple cases.
Theorem \ref{qcount} (through Theorem \ref{poinpoly}) is also helpful in understanding the geometry of $\B(\mathsf{x},H(\mm))$.  We obtain the following result in Section~\ref{reducible.sec}.

\begin{THM} \label{irreducible}
For $\mathsf{x} \in \g$, the variety $\B(\mathsf{x},H(\mm))$ is reducible if and only if there is some $\lambda \in \C$ such that $\mathsf{x}-\lambda I_n$ has rank one.
\end{THM}

It follows immediately from Theorem~\ref{irreducible} that $\B(\mathsf{x},H(\mm))$ is reducible if and only if $\mathsf{x}$ is nilpotent with Jordan decomposition corresponding to the partition $(2,1^{n-2})$ or $\mathsf{x}$ is semisimple and conjugate to a diagonal matrix $\mathrm{diag}(c_1, \ldots, c_1, c_2)$ such that $c_1, c_2\in \C$ with $c_1\neq c_2$.

The proof of Theorem \ref{irreducible} relies on Theorem \ref{qcount} and the examination of affine patches.  Let $B_-$ be the Borel subgroup consisting of all lower triangular matrices in $G$ and let $U_-\simeq \C^{{n}\choose{2}}$ be the unipotent radical of $B_-$, consisting of all matrices $I+\nilp$ with $\nilp$ strictly lower triangular.  The \textit{patch} in $\B$ centered at a point $gB$ is $\NN_g:=gB_-B/B$.  Each coset in this patch has a unique representative of the form $guB$ with $u \in U_-$.  It follows that $\NN_g$ is open in $\B$ and isomorphic to affine space $\C^{{n} \choose {2}}$.  Now for $gB \in \B(\mathsf{x},H(\mm))$, we see that $\NN_{g,\mathsf{x}}^{\mm}:=\NN_g \cap \B(\mathsf{x},H(\mm))$ is an open neighborhood of $gB$ in $\B(\mathsf{x},H(\mm))$ that is isomorphic to an affine hypersurface in $U_-$ cut out by a certain determinant.  Using this fact, we will show that $\B(\mathsf{x},H(\mm))$ is equidimensional.  This allows us to determine whether or not $\B(\mathsf{x},H(\mm))$ is irreducible by calculating its top Betti number, which Theorem \ref{poinpoly} renders easy.

Examination of patches enables us to determine the singular locus of $\B(\mathsf{x},H(\mm))$ when $\mathsf{x}$ is nilpotent.  Indeed, the point $gB$ is smooth in $\B(\mathsf{x},H(\mm))$ if and only if it is smooth in $\NN_{g,\mathsf{x}}^{\mm}$.  Direct calculation yields the next result, which we prove in Section \ref{singular.sec}.  We remark that this method was used to determine the singular loci of the members of another class of Hessenberg varieties, the Petersen varieties, by Insko and Yong in \cite{Insko-Yong2012}. 
These patches were also used by Abe, Dedieu, Galetto and Harada in \cite{ADGH}, by Abe, Fujita and Zeng in \cite{AFZ}, and by Da Silva and Harada in \cite{DH}.

\begin{THM} \label{nilpotent}
%If $\mathsf{x} \in \g$ is nilpotent, then the singular locus of $\B(\mathsf{x},H(\mm))$ is $\B(\mathsf{x},H((1,n-1,\ldots,n-1,n)))$.
The singular locus of $\B(\mathsf{x},H(\mm))$ is contained in $\B(\mathsf{x},H((1,n-1,\ldots,n-1,n)))$. If $\mathsf{x}\in\g$ is nilpotent, then this containment is an equality.
\end{THM}

We consider also the scheme theoretic structure of $\B(\mathsf{x}, H(\mm))$ defined as in~\cite[Section 4]{ITW2020}.
Use of patches yields the following result, which we prove in Section~\ref{sec.reduced}. 

\begin{THM} \label{reduced}
Suppose $n\ge 3$.
For any $\mathsf{x} \in \mathfrak{gl}_n(\C)$ the scheme $\B(\mathsf{x},H(\mm))\subseteq GL_n(\C)/B$ is reduced.
\end{THM}

Combining Theorems \ref{qcount}, \ref{poinpoly} and \ref{nilpotent}, we settle a problem we studied in \cite{EPS}: we determine which Schubert varieties of codimension one in $\B$ are isomorphic to some Hessenberg variety in $\B$.  Write $w_0$ for the longest element in the Weyl group $W=S_n$ of $G$.  For $1 \leq i \leq n-1$, let $s_i$ denote the simple reflection $(i,i+1) \in S_n$.  Given $w \in S_n$, the Schubert variety $X_w \subseteq \B$ has codimension one if and only if $w=s_iw_0$ for some $i \in [n-1]$.  It follows from work of Tymoczko in \cite{Tymoczko2006A} that both $X_{s_1w_0}$ and $X_{s_{n-1}w_0}$ are Hessenberg varieties in $\B$.  We showed in \cite{EPS} that if $3 \leq i \leq n-3$ then no Hessenberg variety in $\B$ is isomorphic to $X_{s_iw_0}$.  (Both Tymoczko's result and ours do not involve the assumption $\b \subseteq H$.)  In Section \ref{schubert.sec}, we obtain the following extension of these results.

\begin{COR} \label{schubert}
If $i \in \{2,n-2\}$ then no Hessenberg variety in $\B$ is isomorphic to $X_{s_iw_0}$.  Therefore, for $j \in [n-1]$, the three conditions
\begin{itemize}
\item some Hessenberg variety in $\B$ is isomorphic to $X_{s_jw_0}$;
\item $X_{s_jw_0}$ is a Hessenberg variety; and
\item $j \in \{1,n-1\}$
\end{itemize}
are equivalent.
\end{COR}

One can compute $\poin(X_{s_iw_0};q)$ directly and use Theorem \ref{poinpoly} to show that, for $i \in \{2,n-2\}$,  $\poin(\B(\mathsf{x},H(\mm));q)=\poin(X_{s_iw_0};q)$ if and only if $\mathsf{x}$ is nilpotent of rank two.  For such $\mathsf{x}$ we compute $|\B_p(\bar{\mathsf{x}}^\prime ,\bar{H}((1,n-1,\ldots,n-1,n)))|$ and so determine the Poincar\'e polynomial of the singular locus $\B(\mathsf{x},H((1,n-1,\ldots,n-1,n)))$ of $\B(\mathsf{x}, \mm)$ using Theorems \ref{qcount} and~\ref{nilpotent}.  A beautiful result conjectured by Lakshmibai and Sandhya in \cite{LS} and proved, independently, by Manivel in \cite{Man}, by Kassel-Lascoux-Reutenauer in \cite{KLR}, and by Billey-Warrington in \cite{BW} allows us to determine the singular loci of the two $X_{s_iw_0}$ in question. We complete the proof of Corollary \ref{schubert} by comparing the Poincar\'e polynomials of these singular loci with those of $\B(\mathsf{x},H((1,n-1,\ldots,n-1,n)))$ for $\mathsf{x}$ nilpotent of rank two.

%%%%%%%%%%%%%%%%

\

\textbf{Acknowledgements:} 
Escobar is partially supported by NSF Grant DMS 1855598 and NSF CAREER Grant DMS 2142656. Precup is partially supported by NSF Grant DMS 1954001 and NSF CAREER Grant DMS 2237057.  Shareshian was partially supported by NSF Grant DMS 1518389.

%%%%%%%%%%%%%%%%
%%%%%%%%%%%%%%%%

\section{Notation and preliminaries}\label{prelim}

We review here various known results about the general linear group, flag varieties, and
Hessenberg varieties. A reader familiar with basic facts about these objects can skip most of this section
and refer back when necessary, but may wish to consult Remark~\ref{rem.notation} which sets notation for the remainder.

\subsection{Basic notation} The symbol $\N$ will denote the set of positive integers.  For $n \in \N$, let $[n]: = \{1, 2, \ldots, n\}$ and $S_n$ denote the symmetric group on $[n]$.  For an integer partition $\lambda=(\lambda_1,\ldots,\lambda_\ell)$, $|\lambda|$ will denote the value $\sum_{j=1}^\ell \lambda_j$.

\subsection{The flag variety} Let $n \in \N$ and $\ff$ be a field. We use two different models for the {\it flag variety} $\B=\B(n,\ff)$.  First, $\B$ consists of all full flags
$$
V_\bullet=\{0=V_0 \subset V_1 \subset \ldots \subset V_n=\ff^n \mid \dim_{\ff} V_i =i \textup{ for all } i\}.
$$
Having fixed a basis $\{e_1,\ldots,e_n\}$ for $\ff^n$, we define the flag $E_\bullet \in \B$ by
$$
E_j=\ff\{e_i \mid i \leq j\}
$$
for each $j \in [n]$.  We use the basis $\{e_1,\ldots,e_n\}$ to coordinatize $GL_n(\ff)$.  So, $a=(a_{ij}) \in GL_n(\ff)$ maps $e_j$ to $\sum_{i=1}^n a_{ij}e_i$.

The natural action of $G_\ff:= GL_n(\ff)$ on $\ff^n$ determines a transitive action of $G_\ff$ on the set $\B$ of full flags.  The stabilizer of $E_\bullet$ in this action is the Borel subgroup $B_\ff$ of upper triangular matrices in $G_\ff$.  So, the map from the coset space $G_\ff/B_\ff$ to $\B$ sending $gB_\ff$ to $g(E_\bullet)$ is a bijection.  We make no distinction between $\B$ and $G_\ff/B_\ff$, other than referring to the set of flags as the {\it flag model} for $\B$ and $G_\ff/B_\ff$ as the {\it coset model}.  We use whichever model suits our purposes at any time.

\subsection{The symmetric group and Bruhat order}

We denote elements of $S_n$ using one-line notation, $$w=[w_1, w_2, \ldots,  w_n]\in S_n$$ where $w_i=w(i)$ in the natural action of $w$ on $[n]$.  For $i \in [n-1]$, we write $s_i$ for the \textit{simple reflection} exchanging $i$ and $i+1$ and fixing all $j \in [n] \setminus \{i,i+1\}$.  Each $w \in S_n$ is a product of simple reflections,
$$
w=s_{i_1}\ldots s_{i_\ell},
$$
with each $i_j \in [n-1]$.  Given such a product with $\ell$ as small as possible, we say that $w$ has \textit{length} $\ell$ and write $\ell(w)=\ell$.

An \textit{inversion} of $w$ is a pair $(i,j) \in [n] \times [n]$ such that $i<j$ and $w_i>w_j$.  We write $\Inv(w)$ for the set of inversions of $w$ and $\invv(w)$ for $|\Inv(w)|$.  It is well known that $\ell(w)=\invv(w)$.

An expression of $w \in S_n$ as a product of $\ell(w)$ simple reflections is a {\it reduced word} for $w$.  The \textit{Bruhat order} $\leq_\br$ on $S_n$ is the partial order in which $v \leq_\br w$ if some reduced word for $v$ is a (not necessarily consecutive) subword of some reduced word for $w$.

The {\it tableau criterion} gives another characterization of the Bruhat order.  Given $w \in S_n$ and $1 \leq p \leq q \leq n$, we write $I_{p,q}(w)$ for the $p^{th}$ smallest element of $\{w_i \mid i \in [q]\}$.  So, for example, if $w=[5,2,3,4,1] \in S_5$, then $I_{2,4}(w)=3$, as $3$ is the second smallest element of $\{5,2,3,4\}$.  The following result of Ehresmann appears as~\cite[Theorem 2.6.3]{Bjorner-Brenti}.

\begin{theorem}[Tableau Criterion] \label{tc}
Let $v,w\in S_n$.  Then $v \leq_\br w$ if and only if $I_{p,q}(v) \leq I_{p,q}(w)$ whenever $1 \leq p \leq q \leq n$.
\end{theorem}

\subsection{Schubert cells and Schubert varieties} Given $w \in S_n$, we write $\dot{w}$ for the element of $G_\ff$ mapping $e_j$ to $e_{w_j}$ for each $j \in [n]$.  Now the \textit{Bruhat decomposition} states that $G_\ff$ is the disjoint union of $B_\ff-B_\ff$ double cosets, each containing exactly one $\dot{w}$,
$$
G_\ff=\bigsqcup_{w \in S_n}B_\ff \dot{w} B_\ff.
$$
We obtain a corresponding decomposition of the flag variety
\begin{equation}\label{eq_Bruhat_dec}
G_\ff/B_\ff=\bigsqcup_{w \in S_n}(B_\ff \dot{w} B_\ff)/B_\ff.
\end{equation}
The \textit{Schubert cell} associated to $w \in S_n$ is
$$
C_w:=(B_\ff \dot{w} B_\ff)/B_\ff.
$$
The set $C_w$ is isomorphic to affine space $\ff^{\ell(w)}$.  The \textit{Schubert variety} associated to $w$ is
$$
X_w:=\bigcup_{v \leq_\br w}C_v.
$$
If $\ff$ is algebraically closed then $X_w$ is the Zariski closure of $C_w$ in $\B$.

\subsection{Hessenberg varieties in the flag and coset models} Having fixed $n \in \N$, we write $\gg_\ff$ and $\bb_\ff$, respectively, for the $\ff$-Lie algebras of all $n \times n$ matrices and all $n \times n$ upper triangular matrices over $\ff$.  These are the respective Lie algebras of $G_\ff$ and $B_\ff$.

The \textit{adjoint representation} is the action of $G_\ff$ on $\gg_\ff$ by conjugation.  
A \textit{Hessenberg space} is a subspace of $\gg_\ff$ that is $B_\ff$-invariant under this action.  
Given $\mathsf{x} \in \gg_\ff$ and Hessenberg space $H \subseteq \gg_\ff$, the associated {\it Hessenberg variety} is
$$
\B(\mathsf{x},H):=\{gB_\ff \in G_\ff/B_\ff \mid g^{-1}\mathsf{x}g \in H\}.
$$
We assume throughout that $\bb_\ff\subseteq H$.  As explained in the Introduction, each such Hessenberg space is of the form $H=H(\m)$ for some Hessenberg vector $\m=(\m(1),\ldots, \m(n))$. It is straightforward to see that in the flag model,
$$
\B(\mathsf{x},H(\m))=\{V_\bullet \mid \mathsf{x}V_i \subseteq V_{\m(i)} \mbox{ for all } i \in [n]\}.
$$
The flag and coset models of $\B(\mathsf{x}, H(\m))$ define the same scheme-theoretic structure on $\B(\mathsf{x}, H(\m))$ by~\cite[Theorem 10]{ITW2020}.

\subsection{Conjugation of Lie algebra elements and translation of Hessenberg varieties} We will use without comment that fact that if $g \in G_\ff$, $\mathsf{x} \in \g_\ff$, and $H \subseteq \g_\ff$ is a Hessenberg space, then
$$
\B(g^{-1}\mathsf{x}g ,H)=g\B(\mathsf{x},H).
$$
In particular, $\B(g^{-1}\mathsf{x}g ,H)$ and $\B(\mathsf{x},H)$ are isomorphic varieties.

\subsection{The number of points in certain finite varieties} We recall some notation and facts that will be useful when $\ff$ is the finite field $\ff_q$ for some prime power $q$.  For each $n \in \N$, write $[n]_t$ for the polynomial $\sum_{j=0}^{n-1}t^j$ and set
$$
[n]_t!:=\prod_{j=1}^n [j]_t.
$$
It is straightforward to compute that for every prime power $q$,
\begin{eqnarray}\label{eqn.proj.count}
|\P^n(\ff_q)|=[n]_q
\end{eqnarray}
and
\begin{equation}\label{eqn.flag.count}
|G_{\ff_q}/B_{\ff_q}|=[n]_q!.
\end{equation}
A consequence of the Bruhat decomposition \eqref{eq_Bruhat_dec} is that
	\begin{equation}\label{eq_flag_Poinc}
	\poin(G_{\C}/B_{\C};q^{1/2})
	=
	\sum_{w\in S_n} q^{\ell(w)}
	=
	[n]_q!
	=|G_{\ff_q}/B_{\ff_q}|
	\end{equation}
where the second equality is well known, see e.g.~\cite[Corollary~1.3.13]{Stanley}.
Thus, Theorem~\ref{qcount} can be seen as a generalization of this result.

\subsection{A remark about nilpotent linear transformations} We require the following straightforward result in the proofs of Theorem~\ref{nilpotent} and Proposition~\ref{echess} below.  Here $\ff$ is an arbitrary field.

\begin{remark} \label{nilpotent.image}
If $\mathsf{x} \in M_n(\ff)$ is nilpotent and $V_{n-1} \subseteq \ff^n$ is an $\mathsf{x}$-invariant $(n-1)$-dimensional subspace, then $V_{n-1}$ contains $\Ima(\mathsf{x})$.
\end{remark}

\begin{proof}
Since the hyperplane $V_{n-1}$ is $\mathsf{x}$-stable and $\mathsf{x}$ is nilpotent, it is also nilpotent as a linear transformation on $\ff^n/V_{n-1}\simeq \ff$. Thus it acts as the zero matrix on the quotient, and we conclude  $\Ima(\mathsf{x}) \subseteq V_{n-1}$.  
\end{proof}

\subsection{The general linear vs.~special linear group over $\C$}

Our main results are all stated and proved under the assumption that $G_\C=GL_n(\C)$.  They hold also if we assume $G_\C=SL_n(\C)$.
Recall that $SL_n(\C)$ has Lie algebra $\mathfrak{sl}_n(\C)$ of $n\times n$ matrices with trace $0$.  As $SL_n(\C)$ acts transitively on the set $\{V_\bullet\}$ of all full flags in $\C^n$, there is an isomorphism between the flag varieties $GL_n(\C)/B$ and $SL_n(\C)/(B \cap SL_n(\C))$, where $B$ is the group of nonsingular upper triangular complex matrices.  Let ${\mathfrak s}$ be the set of scalar matrices in $\gl_n(\C)$ and let $\b$ be the Lie algebra of $B$.  The map sending $H$ to $H \cap \sl_n(\C)$ is a bijection from the set of $B$-invariant subspaces of $\gl_n(\C)$ containing $\b$ to the set of $B \cap SL_n(\C)$-invariant subspaces of $\sl_n(\C)$ containing $\b \cap \sl_n(\C)$.  Indeed, the inverse map sends $H^\prime$ to $H^\prime + {\mathfrak s}$. 

One can define Hessenberg varieties in $SL_n(\C)/(B \cap SL_n(\C))$ as they were defined in $GL_n(\C)/B$.  Given $\mathsf{x} \in \gl_n(\C)$, there is some scalar $\lambda$ such that $\mathsf{x}-\lambda I_n \in \sl_n(\C)$.  If $H \subseteq \gl_n(\C)$ is a Hessenberg space containing $\b$, then 
$$
\B(\mathsf{x},H)=\B(\mathsf{x}-\lambda I_n,H) \cong \B(\mathsf{x}-\lambda I_n,H \cap \sl_n(\C)) \subseteq SL_n(\C)/(B \cap SL_n(\C)).
$$
Conversely, if $\mathsf{x} \in \sl_n(\C)$ and $H^\prime \subseteq \sl_n(\C)$ is a $B \cap SL_n(\C)$-invariant subspace containing $\b \cap \sl_n(\C)$, then
$$
\B(\mathsf{x},H^\prime) \cong \B(\mathsf{x},H^\prime+{\mathfrak s}) \subseteq GL_n(\C)/B.
$$
The upshot of all of this is that if one considers only Hessenberg spaces containing a Borel subalgebra, then there is a bijection between the set of Hessenberg varieties in $GL_n(\C)/B$ and the set of Hessenberg varieties in $SL_n(\C)/(B \cap SL_n(\C))$ under which corresponding varieties are isomorphic.  A similar statement holds for any reductive algebraic group with irreducible root system of type $A_{n-1}$.  Matters are more complicated if one does not assume that Hessenberg spaces contain Borel subalgebras but this will not arise in our work below.

We conclude by choosing shorthand notation for the relevant groups for use in Sections~3-8.
\begin{remark}\label{rem.notation}
For the remainder of the paper we use the notation $G$ for $G_\C$ and $\overline{G}$ for $G_{\ff_p}$ and similarly for their subgroups and respective Lie algebras. In addition, $\B=G_\C/B_\C$ and $\B_p:= G_{\ff_p}/B_{\ff_p}$ throughout. In Section~\ref{qcproof.sec} below we write $C_w$ for the Schubert cell $B\dot wB/B$ in $\B$ and $C_{w,p}$ for the Schubert cell $\bar{B}\dot w \bar{B}/\bar{B}$ in~$\B_p$.
\end{remark}

%%%%%%%%%%%%%%%%%%
%%%%%%%%%%%%%%%%%%

\section{Proof of the point count heuristic} \label{qcproof.sec}

A close examination of Tymoczko's paper \cite{Tymoczko2006} will yield a proof of Theorem \ref{qcount}.  
Given a matrix $\mathsf{x}\in \g$ we may assume, up to replacing $\mathsf{x}$ with a conjugate, that $\mathsf{x}=\semi+\nilp$ where $\semi, \nilp\in \g=\mathfrak{gl}_n(\C)$ are commuting matrices with $\nilp$ nilpotent and $\semi$ diagonal.

\begin{definition} Suppose $\lambda=(\lambda^1, \ldots, \lambda^r)$ is a list of partitions such that $|\lambda^1|+\cdots+|\lambda^r|=n$.  We say that $\mathsf{x}=\semi+\nilp \in \g$ \textit{is of type} $\lambda$ if
\begin{enumerate}
\item $\semi$ has exactly $r$ distinct eigenvalues $c_1, \ldots, c_r$,
\item if $\semi_{ii}=\semi_{kk}$ with $i<k$, then $\semi_{jj}=\semi_{kk}$ whenever $i<j<k$,
\item for each $i\in [r]$, the eigenspace $E_i$ of $\semi$ associated to eigenvalue $c_i$ has dimension $|\lambda^i|$, and
\item for each $i\in [r]$, the Jordan form of the restriction of $\nilp$ to $E_i$, denoted herein by $\nilp_{c_i}$, has Jordan blocks whose sizes are given by the parts of $\lambda^i$.
\end{enumerate}
\end{definition} 
Matrices of the same type need not have the same eigenvalues and so need not be conjugate.  We show below, however, that $\poin(\B(\mathsf{x},H(\m));q)=\poin(\B(\mathsf{y},H(\m));q)$ for every Hessenberg vector $\m$ whenever $\mathsf{x}$ and $\mathsf{y}$ have the same type.

We describe how to fix a representative of each conjugacy class of $\g$, as in~\cite{Tymoczko2006}. Given $\mathsf{x}\in \g$ we say that $\mathsf{x}_{ij}$ is a \textit{pivot} of $\mathsf{x}$ if $\mathsf{x}_{ij}\neq 0$ and if all entries below and all entries to its left vanish.   Suppose $\nilp$ is a nilpotent matrix with Jordan blocks whose sizes are given by the parts of partition~$\lambda$.  Take the Young diagram of shape $\lambda$ and label the boxes with $[n]$ starting at the bottom of the leftmost column, incrementing by one while moving up, then moving to the lowest box of the next column and so on.  The {\bf highest form} of $\nilp$ is the matrix with $\nilp_{ij}=1$ if $i$ labels a box immediately to the left of $j$ and $\nilp_{ij}=0$ otherwise.  All nonzero entries of the highest form of $\nilp$ are pivots, and the highest form is an upper triangular representative for the conjugacy class in which the set of columns containing pivots is right justified.   That is, if a pivot appears in column $i$ with $i<n$, then a pivot appears in column $i+1$.

If $\mathsf{x}=\semi+\nilp \in \g$ is in Jordan normal  form and of type $\lambda = (\lambda^1, \ldots, \lambda^r)$, then we may write $\mathsf{x}=\sum_{i=1}^r (\semi_i + \nilp_{c_i})$ where $\semi_i+\nilp_{c_i}$ acts as $\mathsf{x}$ does on the eigenspace $E_i$ and annihilates $E_j$ when $j \neq i$.  Replacing each $\nilp_{c_i}\in \gl(E_i)$ with its highest form, we obtain a fixed representative for the conjugacy class of $\mathsf{x}$.  We say that this representative is in {\bf highest form and permuted Jordan form} (abbreviated below using the acronym {\bf HFPJF}).

\begin{example}\label{ex.HFPJF} Suppose $n=6$ and $\lambda= ((2,2), (2))$. Let $c_1=1$ and $c_2 = -1$.  The highest form of $\nilp_{c_1}$ and $\nilp_{c_2}$ are obtained from the following labelings of the Young diagrams for $(2,2)$ and $(2)$, respectively.
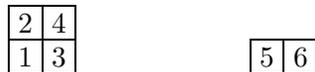
\[
\young(24,13)   \quad\quad \quad\quad \quad\quad \young(56)
\]
The first matrix below in Jordan form and has type $\lambda$; the second is of type $\lambda$ and is the representative for its conjugacy class in HFPJF.
\[
\begin{bmatrix} 1 & 1 & 0 & 0 & 0 & 0\\ 0 & 1 & 0 & 0 & 0 & 0\\ 0 & 0 & 1 & 1 & 0 & 0\\ 0 & 0 & 0 & 1 & 0 & 0\\ 0 & 0 & 0 & 0 & -1 & 1\\ 0 & 0 & 0 & 0 & 0 & -1  \end{bmatrix} \quad\quad\quad\quad 
\begin{bmatrix} 1 & 0 & 1 & 0 & 0 & 0\\ 0 & 1 & 0 & 1 & 0 & 0\\ 0 & 0 & 1 & 0 & 0 & 0\\ 0 & 0 & 0 & 1 & 0 & 0\\ 0 & 0 & 0 & 0 & -1 & 1\\ 0 & 0 & 0 & 0 & 0 & -1  \end{bmatrix}
\]
Notice that in the second matrix the pivots for each $\nilp_{c_i}$ occur in columns as far to the right as possible in each block.
\end{example}

Given a matrix $\mathsf{x} \in \g$ in HFPJF, a permutation $w \in S_n$, and fixed Hessenberg vector $\m$, we define
\begin{eqnarray*}\label{eqn.dim}
d_{w,\mathsf{x}}^\m &:=& \left|\left\{ (i<k) \in \inv(w^{-1}) \mid \begin{array}{c} \textup{$\semi_{ii}=\semi_{kk}$ and if $\nilp_{kj}$ is a pivot of } \\ \textup{$\nilp_{\semi_{ii}}$, then $\m(w^{-1}(j)) \geq w^{-1}(i)$} \end{array}  \right\} \right|\\
&&\quad\quad\quad + |\{ (i<k) \in \inv(w^{-1}) \mid \semi_{ii}\neq \semi_{kk} \textup{ and } \m(w^{-1}(k)) \geq w^{-1}(i)  \}|.
\end{eqnarray*}

The following theorem is~\cite[Theorem 6.1]{Tymoczko2006}.  Note that Tymoczko's conventions for permutation multiplication differ from ours, which is why $w^{-1}$ appears in the formula for $d_{w,\mathsf{x}}^{\m}$ above instead of $w$.

\begin{theorem}[Tymoczko] \label{thm.Tymoczko} Fix a Hessenberg vector $\m$
and suppose $\mathsf{x}=\semi+\nilp$ is 
in HFPJF.  For each $w \in S_n$, $C_{w}\cap \B(\mathsf{x},H(\m))\neq \emptyset$ if and only if $ \nilp\in \dot{w}H(\m)\dot{w}^{-1}$, in which case $C_w\cap \B(\mathsf{x},H(\m))$ is isomorphic to the affine space~$\C^{d_{w,\mathsf{x}}^\m}$. 
\end{theorem}

By Theorem \ref{thm.Tymoczko}, each $\B(\mathsf{x},H(\m))$ is paved by affines. 
Such a paving determines the Betti numbers of any variety.  A key observation is that if $\mathsf{x}$ is in HFPJF and $w \in S_n$, then the nonnegative integer $d_{w,\mathsf{x}}^\m$ depends on the type of $\mathsf{x}$ but not on the eigenvalues $c_i$ of $\mathsf{x}$.  This yields the following result.

\begin{corollary}\label{cor.bettinumber} Suppose $\m$ is a Hessenberg vector and $\mathsf{x}\in \g$ is of type $\lambda$. Let $\mathsf{x}'=\semi+\nilp \in \g$ be the representative for $\mathsf{x}$ in HFPJF. Then
\[
\poin(\B(\mathsf{x},H(\m));q) = \sum_{\substack{w\in S_n\\ \nilp \in \dot{w}H(\m)\dot{w}^{-1}}} q^{2d_{w,\mathsf{x}'}^\m}.
\]
In particular, if $\mathsf{x},\mathsf{y}\in \g$ have the same type, then for every Hessenberg vector $\m$ the Hessenberg varieties $\B(\mathsf{x},H(\m))$ and $\B(\mathsf{y},H(\m))$ have the same Poincar\'e polynomial.
\end{corollary}

We are ready to state and prove a slightly stronger version of Theorem \ref{qcount} after recalling some notation and introducing some additional objects.  Given an integer matrix $\mathsf{x}\in \gl_n(\C)$, we write $\bar{\mathsf{x}}$ for the element of $\gl_n(\ff_p)$ obtained by reducing each entry of $\mathsf{x}$ modulo $p$, $\bar{B}$ for the subgroup of $\bar{G}:=GL_n(\ff_p)$ consisting of upper triangular matrices, and $\B_p$ for the flag variety $\bar G/\bar B$.  A Hessenberg vector $\m$ determines a Hessenberg subspace $\bar{H}(\m)$ of $\gl_n(\ff_p)$ containing the Lie algebra $\bar\fb$ of $\bar B$, and the Hessenberg variety $\B_p(\bar{\mathsf{x}},\bar{H}(\m))$ consists of those $g\bar{B} \in \B_p$ satisfying $g^{-1}\bar{\mathsf{x}}g \in \bar{H}(\m)$.  

Let $U \leq B$ denote the unipotent subgroup of upper triangular matrices with $1$'s on the diagonal. Set  $U_-: = \dot w_0 U \dot w_0^{-1}$, the unipotent subgroup of lower triangular matrices with $1$'s on the diagonal. 
Write $\bar{U}$ and $\bar{U}_-$ for the analogous subgroups of $\bar{G}$. For each $w \in S_n$, consider the subgroup $\bar{U}_w:= \bar{U}\cap \dot w \bar{U}_- \dot w^{-1}$. There is an isomorphism
\begin{eqnarray} \label{eqn.Schubert.cell}
 \bar{U}_w \to C_{w,p} = \bar{B}\dot w\bar{B}/\bar{B},\;  u \mapsto u\dot w \bar{B},
\end{eqnarray}
and $|\bar{U}_{w}| = p^{\ell(w)}$. Finally, given $i \in [n]$, let $\bar{U}_i$ be the subgroup of $\bar{U}$ consisting of all $u \in \bar{U}$ such that $u_{jk}=0$ unless $j \in \{i,k\}$.  

The following lemma gives us a finite field version of \cite[Lemma 5.2]{Tymoczko2006}.

\begin{lemma}\label{lemma.fiber-count}  Let $g\in \bar{U}$ and $\m$ be a Hessenberg vector. Suppose $\mathsf{x}=\semi+\nilp\in \g$ has integer entries and is in HFPJF and $w\in S_n$ satisfies $\nilp \in \dot w H(\m) \dot w^{-1}$.   For each $i \in [n-1]$, set
\begin{eqnarray*}
d_i &:=& \left|\left\{ (i<k) \in \inv(w^{-1}) \mid \begin{array}{c} \textup{$\semi_{ii}=\semi_{kk}$ and if $\nilp_{kj}$ is a pivot of } \\ \textup{$\nilp_{\semi_{ii}}$, then $w^{-1}(i)\leq \m(w^{-1}(j)) $} \end{array}  \right\} \right|\\
&&\quad\quad\quad + |\{ (i<k) \in \inv(w^{-1}) \mid \semi_{ii}\neq \semi_{kk} \textup{ and }\m(w^{-1}(k)) \geq w^{-1}(i)  \}|.
\end{eqnarray*}
If $p$ is a prime such that no two distinct entries of $\mathsf{x}$ have difference divisible by $p$, then for each $i$ the set 
\begin{eqnarray}\label{eqn.fibers}
\{ u\in \bar{U}_{i}\cap \bar{U}_{w} \mid  (u^{-1}g^{-1}\bar{\mathsf{x}} g u)_{ij}=0 \; \textup{ for all } j>i,\; \m(w^{-1}(j)) <w^{-1}(i)  \} 
\end{eqnarray}
is isomorphic to $\mathbb{F}_p^{d_i}$. In particular, it has cardinality $p^{d_i}$.
\end{lemma}
\begin{proof} Our assumptions on the prime $p$ guarantee that $\semi_{ii}\neq \semi_{jj}$ if and only if $\bar{\semi}_{ii}\neq \bar{\semi}_{jj}$. Since $\nilp$ is a $0-1$ matrix we furthermore have that $\nilp$ and $\bar{\nilp}$ have the same pivots, and in particular, that $\nilp_{\semi_{ii}}$ and $\bar{\nilp}_{\semi_{ii}}$ have pivots that occur in the rightmost columns.
Since $g\in \bar{U}$, $g$ is upper triangular with $1$'s on the diagonal.  We have
\[
g^{-1}\bar{\mathsf{x}}g = g^{-1}(\bar{\semi}+\bar{\nilp})g = \bar{\semi} + \nilp'
\]
where $\nilp':= g^{-1}\bar{\nilp}g + g^{-1}\bar{\semi}g-\bar{\semi} \in \gl_n(\mathbb{F}_p)$ is a nilpotent matrix.
Since $\bar{\semi}$ is diagonal, $(g^{-1}\bar{\semi}g-\bar{\semi})_{ij}=0$ whenever $\semi_{ii}=\semi_{jj}$.  In particular, we get that nilpotent matrices $\nilp'_{\semi_{ii}}$ and $\bar{\nilp}_{\semi_{ii}}$ and $\nilp_{\semi_{ii}}$ have pivots in exactly the same entries (cf.~\cite[Proposition 4.6]{Tymoczko2006}).  

By definition,
\[
\bar{U}_{i}\cap \bar{U}_{w} = \left\{ I_n + \sum_{(i\neq\ell)} a_{i\ell}E_{i\ell} \mid a_{i,\ell} \in \mathbb{F}_p,\, a_{i,\ell}=0 \,\textup{ unless }\, (i<\ell)\in \inv(w^{-1}) \right\}.
\]
To complete the proof, we argue that the system of equations defining the set~\eqref{eqn.fibers} is consistent and has $d_i$ free variables.   Let $u = I_n + \sum_{(i\neq\ell)} a_{i,\ell}E_{i,\ell}$ be a generic element of $\bar{U}_{i} \cap \bar{U}_{w}$.  Note that $u^{-1}= I_n -  \sum_{(i\neq\ell)} a_{i,\ell}E_{i,\ell}$.  A straightforward computation now yields 
\begin{eqnarray*}
 (u^{-1}g^{-1}\bar{\mathsf{x}} g u)_{ij} &=&  (u^{-1}\bar{\semi} u)_{ij} + (u^{-1}\nilp'u)_{ij}\\ 
 &=& (\bar{\semi}_{ii}-\bar{\semi}_{jj})a_{ij} + \nilp_{ij}' - \sum_{\substack{(i<\ell) \in \inv(w^{-1})\\ \ell<j}} a_{i\ell}\nilp_{\ell j}'.
\end{eqnarray*}
If $\m(w^{-1}(j)) < w^{-1}(i)$, then $(i<j) \in \inv(w^{-1})$ and using the formula above we see that each of the linear equations defining the set~\eqref{eqn.fibers} is of the form
\begin{eqnarray}\label{eqn.linear}
0= (\bar{\semi}_{ii}-\bar{\semi}_{jj})a_{ij} + \nilp_{ij}' - \sum_{\substack{(i<\ell) \in \inv(w^{-1})\\ \ell<j}} a_{i\ell}\nilp_{\ell j}'.
\end{eqnarray}
If $\semi_{ii}\neq \semi_{jj}$, then~\eqref{eqn.linear} has solution
\[
a_{ij} = (\bar{\semi}_{jj}-\bar{\semi}_{ii})^{-1} \left( \nilp_{i,j}' - \sum_{\substack{(i<\ell) \in \inv(w^{-1})\\ \ell<j}} a_{i\ell}\nilp_{\ell j}' \right).
\]

Now suppose $\semi_{ii}= \semi_{jj}$ and~\eqref{eqn.linear} has at least one nonzero term, namely, there exists $\ell$ with $i\leq \ell<j$ such that $\nilp_{\ell j}'\neq 0$.  Since $\mathsf{x}$ is in HFPJF, all the pivots of $\nilp_{\semi_{ii}}'$ occur in the rightmost columns, so $\nilp'$ and therefore $\nilp$ must have a pivot in column $j$ and row $k$ with $\ell\leq k<j$.  Equivalently, entry $(k,j)$ is a pivot of $\nilp_{\semi_{ii}} = \nilp_{\semi_{kk}}$, and our assumption that $\nilp \in \dot{w}H(\m)\dot{w}^{-1}$ implies $w^{-1}(k)\leq \m(w^{-1}(j))$ so $w^{-1}(k)< w^{-1}(i)$.   Thus $(i<k)\in \inv(w^{-1})$ and $\nilp_{kj}'\neq 0$ so the term $a_{ik}\nilp_{kj}'$ appears in the sum on the RHS of~\eqref{eqn.linear}. This shows that, in the case where $\semi_{ii}=\semi_{jj}$, equation~\eqref{eqn.linear} is not vacuous if and only if there exists $(i<k)\in \inv(w^{-1})$ such that $\nilp_{\semi_{ii}}$ has a pivot in entry $(k,j)$.
Furthermore, in this case, the equation has solution:
\[
a_{ik} = (\nilp_{kj}')^{-1}  \left( \nilp_{i,j}' - \sum_{\substack{(i<\ell) \in \inv(w^{-1})\\ \ell<j, \ell\neq k}} a_{i\ell}\nilp_{\ell j}' \right) .
\]
We have now shown that each equation in the system has a solution, and the system must in fact be consistent since each solution corresponds to either a unique pair $(i<j)\in \inv(w^{-1})$ such that $\semi_{ii}\neq \semi_{jj}$ or a unique pair $(i<k)\in \inv(w^{-1})$ such that $\semi_{ii}=\semi_{kk}$ and entry $(k,j)$ is a pivot of $\nilp_{\semi_{ii}}$ (all pivots must occur in distinct rows).  The number of free variables in the system is equal to the number of inversions $(i<j)$ of $w^{-1}$ such that $\semi_{ii}\neq \semi_{jj}$ and $w^{-1}(i)\leq \m(w^{-1}(j))$ plus
\[
 \left|\left\{ (i<k) \in \inv(w^{-1}) \mid \begin{array}{c} \textup{$\semi_{ii}=\semi_{kk}$ and if $\nilp_{kj}$ is a pivot of } \\ \textup{$\nilp_{\semi_{ii}}$, then $w^{-1}(i)\leq \m(w^{-1}(j)) $} \end{array}  \right\} \right|.
 \]
The proof of the lemma is complete.
\end{proof}

The following implies Theorem~\ref{qcount} and generalizes \eqref{eq_flag_Poinc}, a classical result for the flag variety.

% \begin{theorem} Suppose $\mathsf{x}=\semi+\nilp \in \gl_n(\C)$ has integer entries and is in HFPJF and let $\m$ be a Hessenberg vector.  If $p$ is a prime such that no two distinct entries of $\mathsf{x}$ have difference divisible by $p$, then the cardinality of the Hessenberg variety $\B_p(\bar{\mathsf{x}},\bar{H}(\m)) \subseteq \B_p$ is $\poin(\B(\mathsf{x},H(\m)); p^{1/2})$.
% \end{theorem}

\begin{theorem} Suppose $\mathsf{x}=\semi+\nilp \in \gl_n(\C)$ has integer entries and is in HFPJF and let $\m$ be a Hessenberg vector. Let $p$ be a prime such that no two distinct entries of $\mathsf{x}$ have difference divisible by $p$. Then for each $w \in S_n$ the intersection $C_{w,p}\cap \B_p(\bar{\mathsf{x}}, \bar{H})$ is nonempty if and only if $\mathsf{n}\in \dot w H(\m)\dot w^{-1}$. If nonempty, $C_{w,p}\cap \B_p(\bar{\mathsf{x}}, \bar{H})$ is isomorphic to affine space of dimension $d_{w,\mathsf{x}}^{\m}$. In particular, the cardinality of the Hessenberg variety $\B_p(\bar{\mathsf{x}},\bar{H}(\m)) \subseteq \B_p$ is $\poin(\B(\mathsf{x},H(\m)); p^{1/2})$.
\end{theorem}

% {\color{blue}\textbf{Theorem 3.6 (proposed revision):} Suppose $\mathsf{x}=\semi+\nilp \in \gl_n(\C)$ has integer entries and is in HFPJF and let $\m$ be a Hessenberg vector. Let $p$ be a prime such that no two distinct entries of $\mathsf{x}$ have difference divisible by $p$. Then {\color{purple} for each $w \in S_n$} the intersection $C_{w,p}\cap \B_p(\bar{\mathsf{x}}, \bar{H})$ is nonempty if and only if $\mathsf{n}\in \dot w H(\m)\dot w^{-1}$. If nonempty, $C_{w,p}\cap \B_p(\bar{\mathsf{x}}, \bar{H})$ is isomorphic to affine space of dimension $d_{w,\mathsf{x}}^{\m}$. In particular, the cardinality of the Hessenberg variety $\B_p(\bar{\mathsf{x}},\bar{H}(\m)) \subseteq \B_p$ is $\poin(\B(\mathsf{x},H(\m)); p^{1/2})$.

% }

\begin{proof} Let $w\in S_n$. Since $\nilp$ is a $0-1$ matrix, $\bar{\nilp}_{ij}\neq 0$ if and only if $\nilp_{ij}\neq 0$.  We therefore have $\nilp \in \dot{w} H(\m)\dot{w}^{-1}$ if and only if $\bar{\nilp}\in \dot{w}\bar{H}(\m)\dot{w}^{-1}$.  Furthermore, for every $g\in \bar{U}_{w}$ we get $g^{-1}\bar{\nilp} g = \bar{\nilp} + \mathsf{y}$ where
\[
\mathsf{y} \in \mathbb{F}_p\{ E_{ij} \mid \textup{there exists } (k<\ell) \textup{ such that } i<k, j>\ell,\; \nilp_{k\ell}\neq 0 \}
\]
Thus we may conclude that $g^{-1}\bar{\nilp} g \in \dot{w}\bar{H}(\m)\dot{w}^{-1}$ implies $\bar{\nilp} \in \dot{w}\bar{H}(\m)\dot{w}^{-1}$ since $\bar{\nilp}$ and $\mathsf{y}$ have disjoint support in $\gl_n(\ff_p)$. This proves that 
\begin{eqnarray}\label{eqn.nonempty}
{C}_{w,p}\cap \B_p(\bar{\mathsf{x}},\bar{H}(\m))\neq \emptyset\Leftrightarrow \bar{\nilp} \in \dot{w}\bar{H}(\m)\dot{w}^{-1}  \Leftrightarrow  \nilp\in \dot{w}H(\m)\dot{w}^{-1}.
\end{eqnarray}
We argue now that any nonempty $C_{w,p}\cap \B_p(\bar{\mathsf{x}},\bar{H}(\m))$ is isomorphic with the affine space $\ff_p^{d_{w,{\mathsf x}}^\m}$.  

We adapt Tymoczko's proof of Theorem~\ref{thm.Tymoczko} to the finite field setting (see~\cite[Theorem 6.1]{Tymoczko2006}). Define 
$$
Z_i:=\{u \in (\bar{U}_{n}\bar{U}_{n-1} \ldots \bar{U}_{i}) \cap \bar{U}_{w}\mid (u^{-1}\bar{\mathsf{x}}u)_{jk}=0\, \mbox{ for all }\, k>j\geq i,\; \m(w^{-1}(k))<w^{-1}(j) \}.
$$
If $u\in (\bar{U}_{n}\bar{U}_{n-1} \ldots \bar{U}_{i})\cap \bar{U}_{w}$ then
\begin{eqnarray}\label{eqn.factorization}
u = gu'\, \mbox{ for unique } \,u'\in \bar{U}_{i}\cap \bar{U}_{w}\, \mbox{ and } g\in (\bar{U}_{n}\bar{U}_{n-1}\dots \bar{U}_{i+1})\cap \bar{U}_{w}.
\end{eqnarray}
Furthermore, since conjugation by an element in $\bar{U}_{i}$ affects only the first $i$ rows of an upper triangular matrix, if $u=gu'\in Z_i$ then $g\in Z_{i+1}$.  Thus the factorization from~\eqref{eqn.factorization} yields a well defined projection $\pi_i: Z_i \to Z_{i+1}$ defined by $\pi_i(u) = g$. 
We obtain a sequence of projections
$$
Z_1 \xrightarrow{\; \pi_1\; }Z_2 \xrightarrow{\; \pi_2 \;} \cdots Z_{n-1} \xrightarrow{\; \pi_{n-1}\;} Z_n = \{I_n\}
$$
such that for each $g\in Z_{i+1}$, 
\begin{eqnarray*}
 \pi_i^{-1}(g) = \{ u\in \bar{U}_{i}\cap \bar{U}_{w} \mid (u^{-1}g^{-1}\bar{\mathsf{x}}gu)_{ik} = 0 \, \mbox{ for all }\, k>i, \m(w^{-1}(k))<w^{-1}(i) \}.
\end{eqnarray*}
% {\color{gray}By Lemma~\ref{lemma.fiber-count}, $|\pi_i^{-1}(g)|=p^{d_i}$.  Since $d_{w,\mathsf{x}}^\m=\sum_{j=1}^{n-1}d_j$, we see that $|Z_1| = p^{d_{w,\mathsf{x}}^{\m}}$.
% To complete the proof, note that the bijection $\bar{U}_{w} \to C_{w,p}$ defined as in~\eqref{eqn.Schubert.cell} restricts to a bijection $Z_1\to C_{w,p}\cap \B_p(\bar{\mathsf{x}}, \bar{H}(\m))$. Therefore $|C_{w,p}\cap \B_p(\bar{\mathsf{x}},\bar{H}(\m))| = p^{d_{w,\mathsf{x}}^{\m}}$, as desired. }
By Lemma~\ref{lemma.fiber-count}, the fiber $\pi_i^{-1}(g)$ is isomorphic to affine space of dimension $d_i$. Furthermore, the system of equations from~\eqref{eqn.fibers} defining $\pi^{-1}(g)$ varies continuously in $g$ by conjugation, so $\pi_i$ is a fiber bundle. We have an isomorphism $Z_{i} \to Z_{i+1}\times \mathbb{F}_p^{d_i}$ defined by sending $gu'$ to $(g, v_{u'})$, where $v_{u'}$ is the vector in $\mathbb{F}_p^{d_i}$ determined by the free entries in $u'$ (as defined by the consistent system of equations from~\eqref{eqn.fibers}). Thus, $\pi_i$ is a trivial bundle.

To complete the proof, note that the isomorphism $\bar{U}_{w} \to C_{w,p}$ from~\eqref{eqn.Schubert.cell} restricts to a isomorphism $Z_1\to C_{w,p}\cap \B_p(\bar{\mathsf{x}}, \bar{H}(\m))$. Since $Z_n = \{I_n\}$ and each $\pi_i$ is a trivial bundle, we conclude $C_{w,p}\cap \B_p(\bar{\mathsf{x}},\bar{H}(\m))$ is isomorphic to affine space of dimension $d_{w,\mathsf{x}}^\m=\sum_{i=1}^{n-1}d_i$. This immediately implies $|C_{w,p}\cap \B_p(\bar{\mathsf{x}},\bar{H}(\m))| = p^{d_{w,\mathsf{x}}^{\m}}$, as desired. 
\end{proof}

%%%%%%%%%%%%%%%%%%
%%%%%%%%%%%%%%%%%%

\section{Proofs of Proposition \ref{eigenspaces} and Theorem \ref{poinpoly}} \label{counting.sec}

Recall that Proposition~\ref{eigenspaces} says that if $\mathsf{y} \in \gl_n(\ff_p)$ fixes exactly $k$ $1$-dimensional subspaces of $\ff_p^n$, then
\begin{equation} \label{eigen.eq}
\left|\B_p(\mathsf{y},\bar{H}(\mm))\right|=[n-2]_p!\left([n]_p[n-2]_p+kp^{n-2}\right).
\end{equation}

To prove~\eqref{eigen.eq} we count full flags $0\subset V_1\subset\ldots\subset V_{n-1}\subset\ff_p^n$ in $\B_p$ satisfying 
\begin{equation} \label{hesscon}
\mathsf{y}V_1\subset V_{n-1}
\end{equation}
using the formulas~\eqref{eqn.proj.count} and~\eqref{eqn.flag.count}.
For each $1$-dimensional subspace $Z\subset\ff_p^n$, we count the number of flags that satisfy $V_1=Z$ and~\eqref{hesscon}.  If $Z$ is an eigenspace for $\mathsf{y}$, then any flag satisfying $V_1=Z$ satisfies \eqref{hesscon} and we may pick $V_i$ for $i\geq 2$ such that $Z\subset V_2\subset\ldots\subset V_{n-1}\subset \ff_p^n$ in any way we want.  In particular, $V_2/Z\subset \ldots\subset V_{n-1}/Z\subset \ff_p^n/Z \simeq \ff_p^{n-1}$ is a complete flag in $\ff_p^{n-1}$ every element of the flag variety of $GL_{n-1}(\ff_p)$ arises in this way.  It follows that the number of flags satisfying \eqref{hesscon} in which $V_1$ is an eigenspace for $\mathsf{y}$ is
\begin{equation} \label{eigen1}
k[n-1]_p!=[n-2]_p!k[n-1]_p.
\end{equation}
If $Z$ is not an eigenspace for $\mathsf{y}$, every full flag satisfying \eqref{hesscon} and $V_1=Z$ is obtained by first picking an $(n-1)$-dimensional subspace $V\subset \ff_p^n$ such that $Z+\mathsf{y}Z \subseteq V$, and then picking $V_i$ $(2 \leq i \leq n-2)$ such that $Z\subset V_2\subset \ldots\subset V_{n-2}\subset V$.  Picking $V$ amounts to choosing a hyperspace in $\ff_p^n/Z+\mathsf{y}Z \simeq \ff_p^{n-2}$, so there are $[n-2]_p$ many choices.  By similar reasoning as in the previous case we get that there are $[n-2]_p!$ ways to pick the $V_i$.
We see now that the number of flags satisfying \eqref{hesscon} in which $V_1$ is not an eigenspace for $\mathsf{y}$ is
\begin{equation} \label{eigen2}
\left([n]_p-k\right)[n-2]_p[n-2]_p!.
\end{equation}
Now (\ref{eigen.eq}) follows from (\ref{eigen1}), (\ref{eigen2}), and straightforward calculation.

Recall also that Theorem \ref{poinpoly} says that if $\mathsf{x} \in \g$ has exactly $\ell$ pairwise distinct eigenvalues $\lambda_1,\ldots \lambda_\ell$ and $\dim_\C\ker(\mathsf{x}-\lambda_jI)=d_j$ for each $j \in [\ell]$, then the Poincar\'e polynomial of $\B(\mathsf{x},H(\mm))$ is
\begin{equation} \label{poinpolyform2}
\poin(\B(\mathsf{x},H(\mm));q)=[n-2]_{q^2}!\left([n]_{q^2}[n-2]_{q^2}+q^{2n-4}\sum_{j=1}^\ell [d_j]_{q^2}\right).
\end{equation}

Given such  $\mathsf{x}$, there exist $\mathsf{x}^\ast \in \g$ in HFPJF and $g \in G$ such that
$$
\mathsf{x}^\ast=g^{-1}\mathsf{x}g
$$
and thus $\B(\mathsf{x},H(\mm))$ and $\B(\mathsf{x}^\ast,H(\mm))$ have the same Poincar\'e polynomial.  
Write $\mathsf{x}^\ast=\semi+\nilp$ with $\semi$ diagonal and $\nilp$ a strictly upper triangular $0-1$ matrix.  
Let $\semi^\prime$ be the diagonal matrix obtained from $\semi$ by replacing each entry $\lambda_i$ with $i$, for all $i \in [\ell]$.  Set $\mathsf{x}^\prime:=\semi^\prime+\nilp$.  So, $\mathsf{x}^\prime$ has integer entries and is in HFPJF and has eigenvalues $1,\ldots,\ell$ with respective multiplicities $d_1,\ldots,d_\ell$.  
By Corollary~\ref{cor.bettinumber}, $\B(\mathsf{x}^\ast,H(\mm))$ and $\B(\mathsf{x}^\prime,H(\mm))$ have the same Poincar\'e polynomial.  Now fix $p>2\ell$ and let $\bar{\mathsf{x}}^\prime  \in \gl_n(\ff_p)$ be as in Theorem \ref{qcount}.  We observe that $\bar{\mathsf{x}}^\prime $ has eigenvalues $1,\ldots,\ell \in \ff_p$ with respective multiplicities $d_1,\ldots,d_\ell$.  It follows that there are exactly $\sum_{j=1}^\ell [d_j]_p$ many $1$-dimensional subspaces of $\ff_p$ fixed by $\bar{\mathsf{x}}^\prime $.  Now \eqref{poinpolyform2} follows from Theorem \ref{qcount} and Proposition \ref{eigenspaces}.

%%%%%%%%%%%%%%%
%%%%%%%%%%%%%%%

\section{Proof of Theorem \ref{nilpotent}} \label{singular.sec}

Our aim is to show that if $\mathsf{x} \in \g$ is nilpotent then the singular locus of $\B(\mathsf{x},H(\mm))$ is $\B(\mathsf{x},H((1,n-1,\ldots,n-1,n)))$.  We begin by restating basic facts that were presented in the introduction.  The unipotent radical $U_-$ of the Borel subgroup $B_-$ opposite to $B$ consists of all lower triangular matrices with diagonal entries equal to $1$.  As a variety, $U_-$ is an affine space with coordinates $z_{ji}$, $1 \leq i<j \leq n$.  

We observe that $U_-B/B$ is open in $\B$, as its complement consists of those $gB$ such that some minor of $g$ obtained from a northwest justified submatrix is zero.  It follows that for a subvariety ${\mathcal V} \subseteq \B$ and $gB \in {\mathcal V}$, the intersection $gU_-B/B \cap {\mathcal V}$ is an open neighborhood of $gB$ in ${\mathcal V}$.  With this in mind, given $\mathsf{x} \in \g$ we define as in the Introduction
$$
\NN_{g,\mathsf{x}}^\m:=gU_-B/B \cap \B(\mathsf{x},H(\m))
$$
for each $gB \in \B(\mathsf{x},H(\m))$.  
Since smoothness is a local property we have that given $\mathsf{x} \in \g$, the Hessenberg variety $\B(\mathsf{x},H(\m))$ is smooth at a point $gB$ if and only if the quasiprojective variety $\NN_{g,\mathsf{x}}^\m$ is smooth at $gB$.

We aim now to understand $\NN_{g,\mathsf{x}}^{\mm}$.  Given the set $$\zzz:=\{z_{ji}\mid1 \leq i<j \leq n\}$$ of 
variables, we define $u \in M_{n,n}(\C[\zzz])$ by
\begin{eqnarray}\label{eqn.u}
u:=\lb u_1 | u_2| \ldots | u_n \rb
\end{eqnarray}
with columns
$$
u_i=e_i+\sum_{j=i+1}^n z_{ji}e_j
$$
for each $i\in [n]$.  Any matrix $a=[a_1|a_2|\ldots|a_n] \in U_-$ is obtained from $u$ by substituting some $a_{ji} \in \C$ for each $z_{ji}$. 

Assume that $gB \in \B(\mathsf{x},H(\mm))$.  Given $a \in U_-$ as above,  consideration of the flag model for $\B$ yields quickly that
$$
gaB \in \B(\mathsf{x},H(\mm)) \mbox{ if and only if } \mathsf{x}ga_1 \in \C\{ga_i \mid i \in [n-1]\}.
$$
Equivalently,
$$
gaB \in \B(\mathsf{x},H(\mm)) \mbox{ if and only if } \det([\mathsf{x}ga_1|ga_1|\ldots |ga_{n-1}])=0.
$$
We define
\begin{eqnarray}\label{eqn.Ag}
A_g:=\lb \mathsf{x}gu_1|gu_1 | \ldots|gu_{n-1}\rb \in M_{n,n}(\C[\zzz])
\end{eqnarray}
and observe that $\det(A_g)\in \C[\mathbf{z}]$.  Moreover, $\det(A_g)=0$ if and only if ${\mathsf x}$ is a scalar matrix.  Our identification of $gU_-B/B$ with affine space yields the following.

\begin{lemma} \label{patchhyper}
Assume ${\mathsf x}$ is not a scalar matrix.  If $gB \in \B(\mathsf{x},H(\mm))$ then there is an isomorphism of varieties from $\NN_{g,\mathsf{x}}^{\mm}$ to the affine hypersurface ${\mathcal V}(\det(A_g))$, which sends $gB$ to the zero vector.  Therefore, $gB$ is a smooth point in $\B(\mathsf{x},H(\mm))$ if and only if $(0,\ldots,0)$ is a smooth point in ${\mathcal V}(\det(A_g))$.
\end{lemma}

Now we prove Theorem \ref{nilpotent}.

\begin{proof}[Proof of Theorem~\ref{nilpotent}]
Fix $gB \in \B(\mathsf{x},H(\mm))$. By Lemma~\ref{patchhyper} and the Jacobian criterion for smoothness, $gB$ is a smooth point if and only if $\det(A_g)$ has a nonzero linear term.  Write
$$
g=[v_1|\ldots |v_n].
$$
So, $v_i=ge_i$ for all $i \in [n]$.  Now, for each $i$, set
$$
v_i^\ast:=\sum_{j=i+1}^nz_{ji}v_j \in M_{n,1}(\C[\zzz]).
$$
We observe that
$$
A_g=[\mathsf{x}v_1+\mathsf{x}v_1^\ast | v_1+v_1^\ast | \ldots | v_{n-1}+v_{n-1}^\ast].
$$ 
Using multilinearity, we write $\det(A_g)$ as the sum of $2^n$ terms of the form
\begin{equation} \label{multi}
\det([t_0 | t_1 | \ldots | t_{n-1}]),
\end{equation}
where $t_0 \in \{\mathsf{x}v_1,\mathsf{x}v_1^\ast\}$ and $t_i \in \{v_i,v_i^\ast\}$ for each $i \in [n-1]$.  If there exist $i,j \in [n-1]$ such that $t_i=v_i^\ast$ and $t_j=v_j^\ast$ then the determinant in (\ref{multi}) has no linear term.  
The same holds if $t_0=\mathsf{x}v_1^\ast$ and there is some $i \in [n-1]$ with $t_i=v_i^\ast$.  (Indeed, in both these cases, the matrix $[t_0 | t_1 | \ldots | t_{n-1}]$ contains two columns whose entries are polynomials in $\C[\zzz]$ with no constant term.)
If $t_0=\mathsf{x}v_1$ and $t_i=v_i$ for all $i \in [n-1]$ we have that $\det([t_0 | t_1 | \ldots | t_{n-1}])\in \C$. 
The remaining terms from (\ref{multi}) are
$$
D_0:=\det([\mathsf{x}v_1^\ast | v_1 | v_2 | \ldots | v_{n-1}])
$$
and, for each $i \in [n-1]$,
$$
D_i:=\det([\mathsf{x}v_1 | v_1 | \ldots | v_{i-1} | v_i^\ast | v_{i+1} | \ldots | v_{n-1}]).
$$
Using multilinearity again, we see that
\begin{equation} \label{d0}
D_0=\sum_{j=2}^n z_{j1} \det([\mathsf{x}v_j|v_1| \ldots | v_{n-1}]),
\end{equation}
and
\begin{equation} \label{di}
D_i=z_{ni}\det([\mathsf{x}v_1|v_1| \ldots | v_{i-1} | v_n | v_{i+1} | \ldots | v_{n-1}])
\end{equation}
for each $i \in [n-1]$.

Keeping in mind the bijection from cosets to flags mentioned in Section~\ref{prelim}, we see that $gB \in \B(\mathsf{x},H(\mm))$ implies $\mathsf{x}v_1 \in \C\{v_i \mid i \in [n-1]\}$.  Write
$$
\mathsf{x}v_1=\sum_{k=1}^{n-1} \gamma_k v_k,
$$
with each $\gamma_i \in \C$.  We observe that, for each $i \in [n-1]$, $\gamma_i \neq 0$ if and only if $D_i \neq 0$. 
Similarly, $D_0=0$ if and only if $\mathsf{x}v_j \in \C\{v_i\mid i \in [n-1]\}$ whenever $2 \leq j \leq n$. 
The only variable that appears more than once in the $D_j$'s, $j\in \{0,1,\ldots, n-1\}$, is $z_{n1}$, which appears in both $D_0$ and~$D_1$.  We conclude that $\det(A_g)$ contains a linear term whenever $D_i\neq 0$ for some $2\leq i\leq n-1$ or whenever $\det([\mathsf{x}v_i|v_1| \ldots | v_{n-1}])\neq 0$ for some $2\leq i\leq n-1$. 

To prove the first statement of the theorem, assume that $gB$ is singular. Since $\det(A_g)$ has no linear term,  we have $D_i=0$ for all $2\le i \le n-1$ and it follows that 
\begin{equation}\label{eq_thm5_1}
\mathsf{x}v_1\in \C\{v_1\}.
\end{equation}
We also have that $\det([\mathsf{x}v_i|v_1| \ldots | v_{n-1}])=0$ for all $2\le i \le n-1$ which, combined with \eqref{eq_thm5_1}, implies that for all $i\in[n-1]$,
\begin{equation}\label{eq_thm5_2}
\mathsf{x}(\C\{v_j \mid j\in[i]\})\subseteq \C\{v_k \mid k \in [n-1]\}.
\end{equation}
The bijection from cosets to flags gives us that $gB\in \B(\mathsf{x},H((1,n-1,\ldots,n-1,n)))$ if and only if \eqref{eq_thm5_1} and \eqref{eq_thm5_2} hold. We conclude that $gB\in \B(\mathsf{x},H((1,n-1,\ldots,n-1,n)))$.

To prove the second statement of the theorem, assume that $\sf x$ is nilpotent and $gB\in \B(\mathsf{x},H((1,n-1,\ldots,n-1,n)))$.
Since  \eqref{eq_thm5_1} holds and the only eigenvalue of $\mathsf{x}$ is zero, $\mathsf{x}v_1 = 0$. 
It follows that for each $i\in[n-1]$, $\gamma_i=0$ and thus $D_i=0$.
Equation~\eqref{eq_thm5_2} must also hold, so the hyperplane $\C\{v_i \mid i \in [n-1]\}$ is $\mathsf{x}$-stable.  As $\mathsf{x}$ is nilpotent, we have $\mathsf{x}v_n \in \C\{v_i\mid i \in [n-1]\}$ by Remark~\ref{nilpotent.image}.  We conclude $D_0=0$ and $\det(A_g)$ has no linear terms. This concludes the proof of the second statement.
\end{proof}

\begin{example}
Consider $\B(\mathsf{x},H(\mm))$ with $\mathsf{x}=\mathrm{diag}(1,0,0)\in \mathfrak{gl}_3(\C)$.
Let $g$ be the identity matrix and note that $gB$ is a smooth point of this Hessenberg variety since
	$$
	\det(A_g)=\det\begin{bmatrix} 1&1&0\\ 0&z_{21}&1\\ 0&z_{31}&z_{32}
	\end{bmatrix}=z_{21}z_{32}-z_{31}
	$$
has a nonzero linear term.
Since $gB\in\B(\mathsf{x},H((1,n-1,\ldots,n-1,n))) = \B(\mathsf{x}, H((1,2,3)))$ we see that the nilpotent hypothesis is a necessary assumption in the statement of Theorem~\ref{nilpotent}.
\end{example}

%%%%%%%%
%%%%%%%%
%%%%%%%%

\section{Proof of Theorem~\ref{irreducible}} \label{reducible.sec}

Here we prove Theorem~\ref{irreducible}, which characterizes those $\mathsf{x}$ for which $\B(\mathsf{x},H(\mm))$ is irreducible.   
We continue using the notation appearing in the previous section.
It is known (see for example~\cite{ADGH}) that the flag variety $\B$ has open cover
$$
\B=\bigcup_{w \in W} \dot{w}U_-B/B.
$$
We therefore obtain an open cover,
\begin{equation} \label{unionhess}
\B(\mathsf{x},H(\mm))=\bigcup_{w \in W} \NN_{\dot{w},\mathsf{x}}^{\mm}
\end{equation}
of $\B(\mathsf{x}, H(\mm))$.

By the Principal Ideal Theorem~\cite[Theorem 10.1]{Eis}, every affine hypersurface in  $\C^{{n} \choose {2}}$ is equidimensional of codimension one. The Hessenberg variety $\B(\mathsf{x},H(\mm))$ is covered by open neighborhoods, each of which is isomorphic to an affine hypersurface in  $\C^{{n} \choose {2}}$ by Lemma 5.1. Thus $\B(\mathsf{x},H(\mm))$ is equidimensional.

%By Lemma~\ref{patchhyper}, each $\NN_{\dot{w},\mathsf{x}}^{\mm}$ is isomorphic to a hypersurface in $\C^{{n} \choose {2}}$ and is therefore Cohen-Macaulay, which in turn implies equidimensional (see, for example, \cite[Corollary~18.11]{Eis}). 
 %In particular, each irreducible component of $\NN_{\dot{w},\mathsf{x}}^{\mm}$ has codimension one in $\C^{{n} \choose {2}}$.

%Now assume that $\B(\mathsf{x},H(\mm))$ is reducible.  By Theorem \ref{poinpoly}, $\B(\mathsf{x},H(\mm))$ is connected.  (Indeed, $\poin(\B(\mathsf{x},H(\mm));q)$ has constant coefficient $1$.)  Therefore, there exist distinct irreducible components $Y_1$ and $Y_2$ of $\B(\mathsf{x},H(\mm))$ such that $Y_1 \cap Y_2 \neq \emptyset$.  Fix $gB \in Y_1 \cap Y_2$ and let $w \in W$ such that $gB \in \NN_{\dot{w},\mathsf{x}}^{\mm}$.  Each $Y_i \cap \NN_{\dot{w},\mathsf{x}}^{\mm}$ is an irreducible component of $\NN_{\dot{w},\mathsf{x}}^{\mm}$ and therefore has dimension ${{n} \choose {2}}-1$.  It follows that $\dim Y_i={{n} \choose {2}}-1=\dim \B(\mathsf{x},H(\mm))$.  This proves $\B(\mathsf{x},H(\mm))$ is equidimensional.

We conclude that $\B(\mathsf{x},H(\mm))$ is reducible if and only if the degree $$2\dim \B(\mathsf{x}, H(\mm)) =  n^2-n-2$$ polynomial $\poin(\B(\mathsf{x},H(\mm));q)$ is not monic. 
Recall from Theorem \ref{poinpoly} that if $\mathsf{x}$ has eigenvalues $\lambda_1,\ldots,\lambda_\ell$ with respective multiplicities $d_1,\ldots,d_\ell$, then
$$
\poin(\B(\mathsf{x},H(\mm));q)=[n-2]_{q^2}!\left([n]_{q^2}[n-2]_{q^2}+q^{2n-4}\sum_{j=1}^\ell [d_j]_{q^2}\right).
$$
The polynomial $[n-2]_{q^2}![n]_{q^2}[n-2]_{q^2}$ is monic of degree $n^2-n-2$, and $[n-2]_{q^2}!q^{2n-4}$ has degree $n^2-3n+2$.  It follows that $\poin(\B(\mathsf{x},H(\mm));q)$ is not monic if and only if
\begin{equation} \label{eigendims}
\max_{j \in [\ell]} (d_j-1)=n-2.
\end{equation}
We observe that (\ref{eigendims}) holds exactly when $x-\lambda_j I$ has rank one for some $j$, and Theorem~\ref{irreducible} follows.

%%%%%%%%
%%%%%%%%
%%%%%%%%

\section{Proof of Corollary~\ref{schubert}} \label{schubert.sec}

Recall that the simple reflections $s_2$ and $s_{n-2}$ act, respectively, as the transpositions $(2,3)$ and $(n-2,n-1)$ in the natural permutation action of $S_n$.  The longest element $w_0 \in S_n$ maps each $i \in [n]$ to $n+1-i$.  Our aim is to show that if $i \in \{2,n-2\}$ then no Hessenberg variety in $\B$ is isomorphic to the Schubert variety $X_{s_iw_0}$.  We know that $\dim X_{s_iw_0}={{n} \choose {2}}-1$.  Lemma 5.5 of \cite{EPS} says that if $\mathsf{x} \in \g$ and Hessenberg space $H \subseteq \g$ satisfy $\dim \B(\mathsf{x},H)={{n} \choose {2}}-1$, then at least one of 
\begin{itemize}
\item $H=H(\mm)$ or
\item $\mathsf{x}-\lambda I_n$ has rank one for some $\lambda \in \C$
\end{itemize}
must hold.  It follows immediately from \cite[Lemma 5.9]{EPS} that if $i \in \{2,n-2\}$ and $\mathsf{x}-\lambda I_n$ has rank one for some $\lambda \in \C$, then there is no Hessenberg space $H \subseteq \g$ such that $X_{s_iw_0}$ and $\B(\mathsf{x},H)$ are isomorphic.  So, it remains to show that there do not exist $\mathsf{x} \in \g$ and $i \in \{2,n-2\}$ such that $B(x,H(\mm))$ and $X_{s_iw_0}$ are isomorphic. The next result is our first step toward this goal.

\begin{proposition} \label{samepp}
If $i \in \{2,n-2\}$, $\mathsf{x} \in \g$ and $\poin(\B(\mathsf{x},H(\mm));q)=\poin(X_{s_iw_0};q)$, then there is some $\lambda \in \C$ such that $\mathsf{x}-\lambda I_n$ is nilpotent of rank two.
\end{proposition}
\begin{proof} It is shown in~\cite[Example 5.2]{EPS} that
\begin{equation} \label{schupp}
\poin(X_{s_2w_0};q^{1/2})=\poin(X_{s_{n-2}w_0};q^{1/2})=[n-2]_q!\left([n]_q[n-1]_q-q^{2n-3}-q^{2n-4}\right).
\end{equation}

Combining (\ref{schupp}) and Theorem \ref{poinpoly}, we see that if $\mathsf{x} \in \g$ has eigenvalues $\lambda_1,\ldots,\lambda_\ell$ with respective multiplicities $d_1,\ldots,d_\ell$ and $i \in \{2,n-2\}$ with $\poin(\B(\mathsf{x},H(\mm);q))=\poin(X_{s_iw_0};q)$ then
$$
[n]_q[n-2]_q+q^{n-2}\sum_{j=1}^\ell [d_j]_q = [n]_q[n-1]_q-q^{2n-3}-q^{2n-4},
$$
from which it follows that
$$
q^{n-1}+q^{n-2}+\sum_{j=1}^\ell[d_j]_q = [n]_q.
$$
We see now that if $\poin(\B(\mathsf{x},H(\mm));q)=\poin(X_{s_iw_0};q)$ under the given conditions on $\mathsf{x}$ and $i$, then $\ell=1$ and $d_1=n-2$. 
\end{proof}

Since $\B(\mathsf{x},H(\mm))=\B(\mathsf{x}-\lambda I_n,H(\mm))$ for every $\mathsf{x} \in \g$ and every $\lambda \in \C$, it suffices now to show that if $\mathsf{x} \in \g$ is nilpotent of rank two and $i \in \{2,n-2\}$, then $\B(\mathsf{x},H(\mm))$ and $X_{s_iw_0}$ are not isomorphic.  We compare the Euler characteristics of the singular loci of the varieties in question, examining the Schubert varieties first.  The next result was conjectured by Lakshmibai and Sandhya in \cite{LS} and proved, independently, by Manivel in \cite{Man}, by Kassel-Lascoux-Reutenauer in \cite{KLR}, and by Billey-Warrington in \cite{BW}.

\begin{lemma} \label{lsconjecture}
Given $w \in S_n$, let $Z_w$ consist of all $v \in S_n$ satisfying either
\begin{enumerate}
\item There exist $1 \leq i<j<k<\ell \leq n$ and $1 \leq i^\prime<j^\prime<k^\prime<\ell^\prime \leq n$ such that
\begin{enumerate}
\item $w_k<w_\ell<w_i<w_j$;
\item $v_{i^\prime}=w_k$, $v_{j^\prime}=w_i$, $v_{k^\prime}=w_\ell$, $v_{\ell^\prime}=w_j$; and
\item if $v^\prime$ is obtained from $w$ be replacing $w_i,w_j,w_k,w_\ell$ with $w_k,w_i,w_\ell,w_j$, respectively, and $w^\prime$ is obtained from $v$ by replacing $v_{i^\prime},v_{j^\prime},v_{k^\prime},v_{\ell^\prime}$ with $v_{j^\prime},v_{\ell^\prime},v_{i^\prime},v_{k^\prime}$, respectively, then
\begin{equation} \label{pat1}
v^\prime \leq v \leq w^\prime \leq w
\end{equation}
\end{enumerate}
in Bruhat order, or
 \item There exist $1 \leq i<j<k<\ell \leq n$ and $1 \leq i^\prime<j^\prime<k^\prime<\ell^\prime \leq n$ such that
\begin{enumerate}
\item $w_\ell<w_j<w_k<w_i$;
\item $v_{i^\prime}=w_j$, $v_{j^\prime}=w_\ell$, $v_{k^\prime}=w_i$, $v_{\ell^\prime}=w_k$; and
\item if $v^\prime$ is obtained from $w$ be replacing $w_i,w_j,w_k,w_\ell$ with $w_j,w_\ell,w_i,w_k$, respectively, and $w^\prime$ is obtained from $v$ by replacing $v_{i^\prime},v_{j^\prime},v_{k^\prime},v_{\ell^\prime}$ with $v_{k^\prime},v_{i^\prime},v_{\ell^\prime},v_{j^\prime}$, respectively, then
\begin{equation} \label{pat2}
v^\prime \leq v \leq w^\prime \leq w
\end{equation}
\end{enumerate}
in Bruhat order.
\end{enumerate}
Let $Z_{w, \mathrm{max}}$ be the set of Bruhat maximal elements of $Z_w$.  Then the singular locus of $X_w$ is $\bigcup_{v \in Z_{w, \mathrm{max}}}X_v$.
\end{lemma}

\begin{proposition} \label{schubertsingular}
Define $v(2)$ and $v(n-2)$ in $S_n$ by
$$
v(2)_i:=\left\{ \begin{array}{ll} n+1-i & 1 \leq i \leq n-4 \\ 2 & i=n-3 \\ 1 & i=n-2 \\ 4 & i=n-1 \\ 3 & i=n \end{array} \right.
$$
and
$$
v(n-2)_i:=\left\{ \begin{array}{ll} n-2 & i=1 \\ n-3 & i=2 \\ n & i=3 \\ n-1 & i=4 \\ n+1-i & 5 \leq i \leq n. \end{array} \right.
$$
Then
\begin{enumerate}
\item the singular locus of $X_{s_2w_0}$ is $X_{v(2)}$, and
\item the singular locus of $X_{s_{n-2}w_0}$ is $X_{v(n-2)}$.
\end{enumerate}
\end{proposition}

\begin{proof}
We begin with (1).  Let $w=s_2w_0$ and write 
$$
w=[n,n-1,\ldots,5,4,2,3,1]
$$
and
$$
v(2)=[n,n-1,\ldots,5,2,1,4,3]
$$
in one-line notation.  We show first that $v(2) \in Z_w$.  If we set $(i,j,k,\ell)=(i^\prime,j^\prime,k^\prime,\ell^\prime)=(n-3,n-2,n-1,n)$ we see that Lemma \ref{lsconjecture}(2) can be applied, with $v^\prime=v=v(2)$ and $w^\prime=w$.

Now assume $v \in Z_w$.  We observe that $v,w$ must be as in Lemma \ref{lsconjecture}(2), with $(j,k,\ell)=(n-2,n-1,n)$.  It follows that
$$
(v^\prime_{n-2},v^\prime_{n-1},v^\prime_n)=(1,n+1-i,3)
$$
for some $i \in [n-3]$.  By the tableau criterion, $v^\prime \leq v$ implies $1 \in \{v_{n-2},v_{n-1},v_n\}$.  As $v_{j^\prime}=w_\ell=1$, we see that $j^\prime \in \{n-2,n-1,n\}$.  It follows that $(j^\prime,k^\prime,\ell^\prime)=(n-2,n-1,n)$.  Therefore,
$$
(v_{n-2},v_{n-1},v_n)=(w_\ell,w_i,w_k)=(1,n+1-i,3),
$$
for some $i \in [n-3]$.
Applying the tableau criterion again, we get that $v \leq v(2)$.
Indeed, $I_{p,q}(v)\le I_{p,q}(w_0) = I_{p,q}(v(2))$ for $(p,q)\notin\{(1,n-3),(1,n-2),(2,n-2),(1,n-1),(2,n-1)\}$ and the remaining cases follow from our analysis.
So, $Z_{w, \mathrm{max}}=\{v(2)\}$ and (1) follows from Lemma \ref{lsconjecture}.

The proof of (2) is similar: If we set $w=s_{n-2}w_0$ we conclude that $v(n-2) \in Z_w$ by applying Lemma \ref{lsconjecture}(2) with $(i,j,k,\ell)=(1,2,3,4)$.  If $v \in Z_w$, then $v,w$ must be as in Lemma \ref{lsconjecture}(2) with $(i,j,k)=(1,2,3)$.  Now $(v^\prime_1,v^\prime_2,v^\prime_3)=(n-2,n+1-\ell,n)$ for some $4 \leq \ell \leq n$.  By the tableau criterion, $n \in \{v_1,v_2,v_3\}$, hence $(i^\prime,j^\prime,k^\prime)=(1,2,3)$ and $(v_1,v_2,v_3)=(n-2,n+1-\ell,n)$.  It follows that $v \leq v(n-2)$, and this implies (2).
\end{proof}

\begin{corollary} \label{ecshcu}
If $i \in \{2,n-2\}$ and $v(i)$ is as in Proposition \ref{schubertsingular} then the Euler characteristic of $X_{v(i)}$ satisfies
$$
\chi(X_{v(i)})=(n-2)!(n^2-5n+6).
$$
\end{corollary}

\begin{proof}
We obtain the Euler characteristic by counting the number of $u \leq v(i)$ in the Bruhat order. Fix $u=[u_1,\ldots, u_n] \in S_n$. Using the tableau criterion, we see that $u \not \leq v(2)$ if and only if $\{u_{n-1},u_n\} \cap \{1,2\} \neq \emptyset$ and $u \not\leq v(n-2)$ if and only if $\{u_1,u_2\} \cap \{n-1,n\} \neq \emptyset$.  The Corollary follows from a straightforward inclusion-exclusion argument.
\end{proof}

\begin{proposition} \label{echess}
Assume $\mathsf{x} \in \g$ is nilpotent of rank two and consider the Hessenberg vector $\m=(1,n-1,\ldots,n-1,n)$.
\begin{enumerate}
\item If $\mathsf{x}^2=0$, then $\chi(\B(\mathsf{x},H(\m)))=(n-2)!(n^2-5n+8)$.
\item If $\mathsf{x}^2 \neq 0$, then $\chi(\B(\mathsf{x},H(\m)))=(n-2)!(n^2-5n+7)$.
\end{enumerate}
\end{proposition}

\begin{proof}
We use Theorem \ref{qcount}, assuming without loss of generality that $\mathsf{x}$ is in highest form and permuted Jordan form.  So, we fix any prime $p>2$ and count flags $0=V_0\subset V_1\subset \ldots\subset V_n=\ff_p^n$ satisfying
\begin{itemize}
\item[(a)] $V_1 \subseteq  \ker(\bar{\mathsf{x}})$ and
\item[(b)] $\Ima(\bar{\mathsf{x}}) \subseteq  V_{n-1}$.
\end{itemize}
Note that Remark~\ref{nilpotent.image} allows us to replace $\bar{\mathsf{x}}V_{n-1} \subseteq  V_{n-1}$, implied by our choice of $\m$, with (b) as stated.

We observe that $\bar{\mathsf{x}}^2=0$ if and only if $\mathsf{x}^2=0$ and consider first the case $\mathsf{x}^2=0$.  We fix $1$-dimensional $Z\subset \ff_p^n$ and count flags satisfying (a) and (b) in which $V_1=Z$.  If $Z \subseteq \Ima(\bar{\mathsf{x}})$ then we can choose $V_{n-1}$ to be any of the $[n-2]_p$ hyperplanes in $\ff_p^n$ containing the $2$-dimensional subspace $\Ima(\bar{\mathsf{x}})$ and then choose $V_2,\ldots,V_{n-2}$ in any of $[n-2]_p!$ ways.  Since $\dim \Ima(\bar{\mathsf{x}})=2$ there are $[2]_p$ possibilities for $Z$ and it follows that there are
$$
[2]_p[n-2]_p[n-2]_p!
$$
flags satisfying (a) and (b) in which $V_1 \subset \Ima(\bar{\mathsf{x}})$.

Now consider the case of $1$-dimensional $Z\subset \ff_p^n$ such that $Z \subseteq \ker(\bar{\mathsf{x}})$ and $Z \not\subseteq \Ima(\bar{\mathsf{x}})$. Since $\bar{\mathsf{x}}^2=0$, we know that $\Ima(\bar{\mathsf{x}}) \subseteq \ker(\bar{\mathsf{x}})$.  Therefore, there are $[n-2]_p-[2]_p$ possibilities for $Z$.  Having fixed one such $Z$, we can choose $V_{n-1}$ containing $Z+\Ima(\bar{\mathsf{x}})$ in $[n-3]_p$ ways, and then choose $V_2,\ldots,V_{n-2}$ in $[n-2]_p!$ ways.  We see there are
$$
\left([n-2]_p-[2]_p\right)[n-3]_p[n-2]_p!
$$
flags satisfying (a) and (b) in which $V_1 \not\subseteq\Ima(\bar{\mathsf{x}})$.  Direct computation, along with Theorem \ref{qcount}, gives
$$
\poin(\B(\mathsf{x},\m);q)=[n-2]_{q^2}!\left([n-2]_{q^2}[n-3]_{q^2}+(1+q^2)q^{2n-6}\right),
$$
and (1) follows.

The case $\mathsf{x}^2 \neq 0$ is handled similarly.  We observe first that in this case
$$
\dim(\ker(\bar{\mathsf{x}}) \cap \Ima(\bar{\mathsf{x}}))=1.
$$
Arguing as we did above, consider first the case in which $V_1 = \ker(\bar{\mathsf{x}}) \cap \Ima(\bar{\mathsf{x}})$.  Then we calculate that there are
$$
[n-2]_p[n-2]_p!
$$
flags satisfying (a), (b), and $V_1 = \ker(\bar{\mathsf{x}}) \cap \Ima(\bar{\mathsf{x}})$.
On the other hand, if $V_1\not\subseteq \Ima(\bar{\mathsf{x}})$ we get
$$
p[n-3]_p[n-3]_p[n-2]_p!
$$
flags satisfying (a) and (b).  As above, Theorem~\ref{qcount} now yields 
$$
\poin(\B(\mathsf{x},\m);q)=[n-2]_{q^2}!\left([n-2]_{q^2}+q^2[n-3]_{q^2}^2\right)
$$
and (2) follows.
\end{proof}

Propositions~\ref{ecshcu} and~\ref{echess} together imply that the singular loci of the Schubert variety $X_{s_iw_0}$ (for $i\in \{2, n-2\}$) and that of the Hessenberg variety $\B(\mathsf{x}, H(\mm))$ (where $\mathsf{x}$ is nilpotent of rank two) are not isomorphic. 
The proof of Corollary~\ref{schubert} is now complete.

%%%%%%%%%%%%%%%%%
%%%%%%%%%%%%%%%%%
%%%%%%%%%%%%%%%%%

\section{Proof of Theorem~\ref{reduced}} \label{sec.reduced}

We use Lemma~\ref{patchhyper} to show that for $n\ge 3$, the Hessenberg variety $\B(\mathsf{x},H(\mm))$ is a reduced scheme for any $\mathsf{x}\in\g$.  
Recall the definition of $A_g \in M_{n\times n}(\C[\zzz])$ from~\eqref{eqn.Ag}. 
The next result suffices, since $\det(A_g)$ generates the ideal defining each open neighborhood $\mathcal{N}_{g,\mathsf{x}}^{\mm}$ of $gB\in \B(\mathsf{x}, H(\mm))$ (cf.~Lemma~\ref{patchhyper}). 

\begin{proposition} \label{radical}
Assume $n\geq 3$. 
If $gB \in \B(\mathsf{x},H(\mm))$, then $\det(A_g)$ generates a radical ideal in~$\C[\zzz]$.
\end{proposition} 

We will need two preliminary lemmas to prove Proposition \ref{radical}.    Recall that, given a set $S$ of variables, a {\it term order} on $S$ is a total ordering $\preceq$ of the set $\mathcal{M}(S)$ of monomials in the elements of $S$ satisfying, for all $\mu_1, \mu_2 \in \mathcal{M}(S)$ and all $s \in S$, the conditions
\begin{itemize}
\item $\mu_1 \prec \mu_1s$, and
\item $\mu_1s \preceq \mu_2s$ whenever $\mu_1 \preceq \mu_2$.
\end{itemize}
We observe that $1$ is the minimum element of $\mathcal{M}(S)$ with respect to any term order.  

Given a polynomial $f \in \C[S]$, the {\it initial term} $\ii(f)$ of $f$ is the monomial appearing with nonzero coefficient in $f$ that is maximal with respect to $\preceq$.  If $I=(f_1,\ldots, f_k)$ is an ideal in $\C[S]$, the {\it initial ideal} $\ii(I)$ is $(\ii(g):g \in I)$.  It is straightforward to show that if $I=(f)$ is a principal ideal then $\ii(I)=(\ii(f))$.

The next result is well known, see e.g.~\cite[Proposition 3.3.7]{Herzog-Hibi}.

\begin{lemma} \label{initial}
Let $I$ be an ideal in $\C[S]$.  If $\ii(I)$ is square-free with respect to some term order on $\mathcal{M}(S)$, then $I$ is radical.
\end{lemma}

Given a total ordering $<_S$ of the variables in $S$, the associated {\it lexicographic term order} on $\mathcal{M}(S)$ is obtained by setting $\mu_1 \preceq \mu_2$ if and only if, when the variables appearing in $\mu_1=s_1\ldots s_j$ and $\mu_2=t_1\ldots t_k$ are written in weakly decreasing order, either $j \leq k$ and $s_i=t_i$ for all $i \in [j]$ or there is some $r$ such that $s_i=t_i$ for all $i<r$ and $s_r<_S t_r$.

The following lemma is the key technical result of this section.

\begin{lemma} \label{induction}
Let $m\geq n\geq 3$, $\zzz = \{z_{ji} \mid 1\leq i<j\leq n\} \sqcup \{z_{j1} \mid n+1\leq j \leq m\}$, and $\ell_1, \ldots, \ell_n\in \C[z_{21}, \ldots, z_{m1}]$ be (not necessarily homogeneous) linear polynomials.  Let $u\in M_{n,n}(\C[\zzz])$ be as in~\eqref{eqn.u} and define
 \begin{eqnarray}\label{eqn.Ldef}
L:=[\ell|u_1|\ldots |u_{n-1}] \in M_{n,n}(\C[\zzz]) \,\textup{ where }\,  {\ell}:=\sum_{i=1}^n \ell_ie_i \in M_{n,1}(\C[z_{21}, \ldots, z_{m1}]).
\end{eqnarray}
Then 
\begin{enumerate}
\item $\det(L)=0$ if and only if $\ell=cu_1$ for some $c\in \C$.  
\item Moreover, there exists a lexicographic term order $\preceq$ on $\mathcal{M}(\zzz)$ such that $\ii(\det (L))$ is square-free.  
\end{enumerate}
 \end{lemma}

\begin{proof} The backward direction of (1) is obvious.  To prove the forward direction and (2), we proceed by induction on $n$.  

Suppose $n=3$. Now 
\begin{eqnarray}\label{eqn.det3}
\det(L) = \det \begin{bmatrix} \ell_1 & 1 & 0 \\ \ell_2 & z_{21} & 1\\ \ell_3 & z_{31} & z_{32} \end{bmatrix}  = (z_{32}z_{21}-z_{31})\ell_1 -z_{32}\ell_2+\ell_3. 
\end{eqnarray}
Write $\ell_i = a_{i1}+\sum_{j=2}^m a_{ij}z_{j1}$ for $1\leq i \leq 3$. Substituting these expressions into~\eqref{eqn.det3} and collecting like terms we obtain the following formula for $\det(L)$:
\begin{eqnarray}\label{eqn.det3.1}
&& a_{12}z_{32}z_{21}^2+a_{13}z_{32}z_{31}z_{21}+ (a_{11}-a_{22})z_{32}z_{21} -a_{12}z_{31}z_{21}-a_{13}z_{31}^2 -a_{23}z_{32}z_{31} -a_{21}z_{32} \\
\nonumber&& \quad\quad +a_{32}z_{21} + (a_{33}-a_{11})z_{31}+a_{31}  + \left( \sum_{j=4}^m a_{1j}z_{32}z_{21}z_{j1} - a_{1j}z_{31}z_{j1} -a_{2j}z_{32}z_{j1} +a_{3j}z_{j1}\right).
\end{eqnarray}
If \eqref{eqn.det3.1} is the zero polynomial, then $a_{ij} = 0$ whenever $i\neq j$ and $a_{11}=a_{22}=a_{33}$. The forward direction of (1) now follows. Notice that~\eqref{eqn.det3.1} has only two monomial terms containing squares: $z_{32}z_{21}^2$ with coefficient $a_{12}$ and $z_{31}^2$ with coefficient $a_{13}$.  If $a_{13}=0$ and $a_{12}=0$ then (2) is immediate, as~\eqref{eqn.det3.1} is then a sum of square-free monomials. 
By inspection, if $a_{13}\neq 0$ then the lexicographic term order $\preceq$ associated to a total ordering of $\zzz$ in which $z_{32}$ is the maximal element and $z_{31}$ is the maximal element of $\zzz\setminus \{z_{32}\}$ yields $\ii (\det(L)) = z_{32}z_{31}z_{21}$. The final case is $a_{13}= 0$ and $a_{12}\neq 0$.  Taking the lexicographic term order $\preceq$ associated to a total ordering of $\zzz$ in which $z_{21}$ is the maximal element and $z_{31}$ is the maximal element of $\zzz\setminus \{z_{21}\}$ gives $\ii (\det(L)) = z_{31}z_{21}$.  This completes the proof of the base case.

Assume now that $n>3$. For $i,j \in [n]$ write $L(i,j)$ for the matrix obtained from $L$ by deleting the $i^{th}$ row and the $j^{th}$ column.  Expanding along the last column, we see that
\begin{equation} \label{maxeq}
\det(L) = \pm z_{n,n-1}\det(L(n,n)) \mp \det(L(n-1,n))
\end{equation}
Let $\zzz' = \{ z_{ji} \mid 1\leq i<j \leq n-1  \} \sqcup\{ z_{j1} \mid n\leq j \leq m \}$ and note that $L(n,n) \in M_{n-1,n-1}(\C[\zzz'])$.

Suppose first that $\det(L(n,n))=0$.  By the induction hypothesis, there exists $c\in \C$ such that $\ell_1 = c $ and $\ell_i = cz_{i1}$ for $2\leq i \leq n-1$. We observe that, in this case, $\det(L)  = \pm \det(L(n-1,n))= \pm (\ell_n - cz_{n1})$.  If $\det(L)=0$ then $\ell_n = cz_{n1}$, proving (1).  Claim (2) follows immediately from the fact that $\ell_n$ is linear.

Next suppose $\det(L(n,n))\neq 0$.  Since $z_{n,n-1}$ does not appear in any monomial term of $\det(L(n-1,n))$, we have $\det(L)\neq 0$ and (1) is vacuously true. 
By the induction hypothesis, there exists a lexicographic term order $\preceq'$ on $\mathcal{M}(\zzz')$ such that $\iii(\det(L(n,n)))$ is square-free.  Let  $\preceq$ be the  lexicographic term order associated to any total order which respects the total order on $\zzz'$ corresponding to $\preceq'$ and with maximal element $z_{n,n-1}$.  Equation~\eqref{maxeq} now implies $\ii(\det(L)) = z_{n,n-1}\iii(\det (L(n,n)))$. The proof is now complete. 
\end{proof}

We are now ready to prove Proposition~\ref{radical}.  

\begin{proof}[Proof of Proposition~\ref{radical}] Set $\mathsf{y}:=g^{-1}\mathsf{x}g$ and
$$
A'_g:=[\mathsf{y}u_1|u_1|\ldots |u_{n-1}].
$$
Comparing $A_g'$ to the definition of $A_g$ in~\eqref{eqn.Ag}, we observe that $A_g=gA'_g$, hence $\det(A'_g)$ and $\det(A_g)$ generate the same ideal in $\C[\zzz]$.   
Now $A'_g$ is as in Lemma~\ref{induction}, where
\begin{eqnarray*}\label{eqn.ell}
\ell_i = \mathsf{y}_{i1} + \sum_{j=2}^n \mathsf{y}_{ij} z_{j1} \; \textup{ for } \; i=1, 2, \ldots, n.
\end{eqnarray*}
Proposition~\ref{radical} now follows from Lemma~\ref{initial} and~\ref{induction}(2).
\end{proof}

The next example demonstrates that the assumption $n\geq3$ was necessary in the statements of Proposition~\ref{radical} and Lemma~\ref{induction}.  Indeed, the example shows that $\B(\mathsf{x}, \mm)$ is not reduced when $n=2$ and $\mathsf{x}$ is nilpotent. 

\begin{example}
Let $\mathsf{y}=\begin{bmatrix} 0&\mathsf{y}_{12}\\ 0&0
\end{bmatrix}$, where $\mathsf{y}_{12}\in\C^*$.
If $A_g'$ is as in the proof of Proposition~\ref{radical}, we have that $A_g'=\begin{bmatrix} \mathsf{y}_{12}z_{21}&1\\ 0&z_{21}
\end{bmatrix}$
and $\det A_{g}'=\mathsf{y}_{12}z_{21}^2$ does not generate a radical ideal.
\end{example}

\end{document}